\documentclass[11pt,reqno]{amsart}
\usepackage{amssymb}
\usepackage{amsmath}
\usepackage{mathrsfs}
\usepackage{amssymb, amsmath, amsfonts}

\usepackage{mathabx}

\usepackage{color}

\usepackage[hidelinks]{hyperref}
\allowdisplaybreaks
\usepackage[numbers,sort&compress]{natbib}

\makeatletter
\def\tank#1{\protected@xdef\@thanks{\@thanks
 \protect\footnotetext[0]{#1}}}
\def\bigfoot{

 \@footnotetext}
\makeatother

\newcommand{\ea}{\end{array}}

\allowdisplaybreaks
\numberwithin{equation}{section}

\allowdisplaybreaks

\newtheorem{theorem}{Theorem}[section]
\newtheorem{lemma}{Lemma}[section]
\newtheorem{proposition}[theorem]{Proposition}

\newtheorem{condition}[theorem]{Condition}
\newtheorem{cor}[theorem]{Corollary}
\newtheorem{definition}[theorem]{Definition}

\newtheorem{remark}{Remark}

\newtheorem{Hypothesis}[theorem]{Hypothesis}

\newtheorem{rem}{Remark}[section]

\def\beq{\begin{equation}}
\def\nneq{\end{equation}}

\def\bthm{\begin{theorem}}
\def\nthm{\end{theorem}}

\def\blem{\begin{lemma}}
\def\nlem{\end{lemma}}
\def\bprf{\begin{proof}}
\def\nprf{\end{proof}}
\def\bprop{\begin{prop}}
\def\nprop{\end{prop}}
\def\brmk{\begin{rem}}
\def\nrmk{\end{rem}}

\def\bexa{\begin{exa}}
\def\nexa{\end{exa}}
\def\bcor{\begin{cor}}
\def\ncor{\end{cor}}

\def\RR{\mathbb{R}}

\def\EE{\mathbb{E}}

\def\cF{\mathcal{F}}
\def\cC{\mathcal{C}}

\def\cH{\mathcal{H}}

\def\cS{\mathcal{S}}

\def\cN{\mathcal{N}}

\newcommand\HHH{\mathfrak H}

\newcommand\HH{\mathcal H}

\def\ls{{\lesssim}}

\def\e{{\varepsilon}}

\def\FF{\mathcal {F}}

\newcommand{\ep}{\varepsilon}

\newcommand{\1}{{\bf 1}}

\newcommand{\Blc}{\Big(}
\newcommand{\Brc}{\Big)}
\newcommand{\blk}{\big[}
\newcommand{\brk}{\big]}
\newcommand{\Blk}{\Big[}
\newcommand{\Brk}{\Big]}
\newcommand{\lc}{\left(}
\newcommand{\rc}{\right)}
\newcommand{\lk}{\left[}
\newcommand{\rk}{\right]}
\newcommand{\lt}{\left }
\newcommand{\rt}{\right}

\title[LDP for nonlinear stochastic  wave equation  driven by rough noise]{A large deviation principle for nonlinear stochastic  wave equation  driven by rough noise}

\author[R. Li]{Ruinan  Li}
\address[]{Ruinan  Li, School of Statistics and Information,  Shanghai University of International Business and Economics,  Shanghai, 201620,  China. }
\email{ruinanli@amss.ac.cn}

\author[B. Zhang]{Beibei Zhang*}\thanks{*Corresponding author. E-mail: zhangbb@whu.edu.cn}
\address[]{Beibei Zhang, School of Mathematics and Statistics,  Wuhan University,  Wuhan, 430072,
China.}
\email{zhangbb@whu.edu.cn}

\date{}
\begin{document}
\maketitle
\noindent{\bf Abstract}
This paper is devoted to investigating Freidlin-Wentzell's large  deviation principle for  one  (spatial)  dimensional nonlinear stochastic wave equation
 $\frac{\partial^2 u^{\e}(t,x)}{\partial t^2}=\frac{\partial^2 u^{\e}(t,x)}{\partial x^2}+\sqrt{\e}\sigma(t, x, u^{\e}(t,x))\dot{W}(t,x)$,
 where $\dot{W}$ is white in time and fractional in space with Hurst parameter $H\in(\frac 14,\frac 12)$. The variational framework and the modified weak convergence criterion proposed by Matoussi et al.  \cite{MSZ} are adopted here.


\vskip0.3cm

\noindent {\bf MR(2010) Subject Classification}{ 60F10, 60G22.  }
\maketitle

\section{Introduction}

In this paper, we consider the following one (spatial) dimensional nonlinear  stochastic  wave  equation (SWE for short):
\begin{equation}\label{SWE}
\left\{\begin{split}
   &  \frac{\partial^2 u^{\e}(t,x)}{\partial t^2}=\frac{\partial^2 u^{\e}(t,x)}{\partial x^2}+\sqrt{\e}\sigma(t,x,u^{\e}(t,x))\dot{W}(t,x),\quad   t>0,\,
   x\in\RR,\\
   &u^{\e}(0,\cdot)=u_0(x),
  \quad\frac{\partial u^{\e}(0,x)}{\partial t}=v_0(x),
 \end{split}\right.
\end{equation}
where $\e>0$ and the noise $W=\left\{W(t,x),t\geq0,x\in\RR)\right\}$ is  a mean zero Gaussian random  field defined on a complete probability space $(\Omega,\mathcal{F},\mathbb{P})$. We assume that the noise $W$ is a standard Brownian motion in time and  a fractional Brownian motion with Hurst parameter $H\in (\frac{1}{4},\frac{1}{2})$ in the space and $\dot{W}(t,x)=\frac{\partial^2W}{\partial t \partial x}(t,x)$. The  covariance of the noise $W$ is given by
 \begin{align}\label{CovStru}
   \EE\lk W(\varphi)W(\psi)\rk=\int_{\RR_{+}\times \RR} \cF \varphi(t,\xi) \overline{\cF \psi(t,\xi)}\cdot \mu(d\xi) dt,\,\,\varphi,\psi\in\mathcal{D}(\RR_+\times\RR),
 \end{align}
 where $\mathcal{D}(\RR_+\times\RR)$ is the space of real-valued infinitely differentiable functions on $\RR_+\times\RR$ with compact support, and
 $\cF \varphi(t,\cdot)$ is the Fourier transform with respect
to the spatial variable $x$ of the  function $\varphi(t,x)$,  which is defined as
 $$
  \cF \varphi(t,\cdot)(\xi):=\int_{\RR} e^{-i\xi x}\varphi(t,x) dx, \xi\in\RR,
$$
and  the measure $\mu$ is given by
\begin{align}
  \mu(d\xi)=c_{1,H}|\xi|^{1-2H}d\xi,\,\, c_{1,H}=\frac{\Gamma(2H+1)\sin(\pi H)}{2\pi},   \label{e.c1}
\end{align}
where $\Gamma$ is the Gamma function and $\frac 14<H<\frac 12$.
When the noise is general Gaussian which is white in time and satisfies Dalang's condition in \cite{Dalang1999} (namely the spatial parameter $H\geq \frac{1}{2}$), there are some results about the well-posedness  and the properties of the solutions, see e.g., \cite{DalangKMN2009, DalangSM2009,HHN2014}. In the case of $\frac 14<H<\frac 12$, since the Fourier transform of $\mu$ in the space of tempered distributions on $\RR$
is not a locally integrable function,  the noise $W$ is not of the same form as the one considered in \cite{Dalang1999}, and the stochastic integral with respect to $W$ was constructed using different methods. See \cite{BJQ2015,HHLNT2017}  for  more details.

Recently, many authors studied the existence and uniqueness for the solutions  of stochastic partial differential equations (SPDEs for short) driven by a Gaussian noise that is white in time and rough in space, see e.g., \cite{HHLNT2017, HHLNT2018, HW2019, LHW2022, SongSX2020}.
Balan et al. \cite{BJQ2015} considered  the the stochastic heat equation (SHE for short) and the SWE with the  affine diffusion coefficient. Hu et al. \cite{HHLNT2017} and Hu and Wang \cite{HW2019}  investigated the SHE with the general nonlinear coefficient.
Liu et al. \cite{LHW2022} considered the SWE driven by the rough noise in space  with the general nonlinear coefficient and Song et al. \cite{SongSX2020} considered the fractional SWE.  Compared with the SHE, the case of the SWE is much more complicated because of the lack of the semigroup property of the  wave kernel. To overcome this diffucity, Liu et al. \cite{LHW2022} decomposed the wave kernel $G(t-s,x-y)$ to four complicated parts (see \cite[Lemma 3.1]{LHW2022} or lemma \ref{3.1wavekernelexp} below). We use the decomposition technique to  establish some estimates in this paper.

The aim of this paper is to establish a large deviation principle (LDP) for the solution $u^\e$ of Eq.\,\eqref{SWE} as $\e\rightarrow0$.
There are several results of LDPs and moderate deviation principles  for the SWE. For example,  Ortiz-L\'opez and Sanz-Sol\'e \cite{OS2011}  considered LDPs for SWE in spatial dimension three, driven by a Gaussian noise, white in time and with a stationary spatial covariance.
Martirosyan \cite{M2017} considered LDPs for stationary measures of stochastic nonlinear wave equations with smooth white noise.
Martirosyana and  Nersesyan \cite{MN2018} established an LDP for occupation measures
of the stochastic damped nonlinear wave equation. Brze\'zniak et al. \cite{B2022} investigated the LDP for solutions of the (1+1)-dimensional stochastic geometric wave equation in the case of vanishing noise.  Cheng et al. \cite{CLWY2018} considered moderate deviations for a stochastic wave equation in dimension three.

An important approach of investigating the LDP is the well-known weak convergence method (see e.g., \cite{BCD2013, BD, BD2019,  BDM2008, BDM2011, DE1997}). For some relevant LDP results by using the weak convergence method, we refer to
\cite{L, LTZ, RZ2008, XuZhang2009, XZ2018} and the references therein.
There are mainly two difficulties to study the LDP for the solution of  Eq.\,\eqref{SWE}.  The first one comes from
the spatial rough noise with $H\in (\frac{1}{4}, \frac{1}{2})$. The second  difficulty comes from
the lack of the semigroup property of the Green function.
In this paper, we adopt  a new sufficient condition for the LDP (see Condition \ref{cond1} below) which is  proposed by Matoussi, Sabbagh and Zhang \cite{MSZ}. This approach  has been proved to be successful in a wide range of  SPDEs, see e.g., \cite{DongWuZhang2020,  LSZ2020, WanfZhang2021, WuZhai2020}.

The rest of the paper is organised as follows. The definition of the stochastic integral, the notion of the solution for the equation \eqref{SWE} and the functional spaces introduced in \cite{LHW2022} are presented in Section 2.
In Section 3, we recall a general criterion for large deviations based on the weak convergence and state our main result.
Section 4 is devoted to showing the existence and uniqueness of  a  solution to the skeleton equation
associated with the equation (1.1). The large deviation principle for equation (1.1) is proved in Section 5. Finally, some auxiliary results are presented in Appendix.

Some notations and mathematical conventions used in this work are as follows. We always use $C_\alpha$  to denote a constant dependent on the parameter $\alpha$,  which may change from line  to
 line. $A\lesssim B$ ($A\gtrsim B$, resp.) means that $A\leq CB$ ($A\geq CB$, resp.) for some positive universal constant $C$, and $A\simeq B$ if and only if $A\lesssim B$ and $A\gtrsim B$.

\section{Preliminaries}\label{lemmas}
This section is divided into two parts. In Section 2.1 we  briefly recall some necessary concepts
about the noise $W$ and the stochastic integral.  In Section 2.2 we collect some preliminaries about the SWE.
\subsection{Stochastic integral}
We recall  some results from \cite{HHLNT2017, HW2019, LHW2022}.
Let $(\Omega,\mathcal{F},\mathbb{P})$ be a complete probability space and let $H \in (\frac14,\frac12)$ be given and
fixed. Our noise $\dot{W}$ is a zero-mean Gaussian family $\{W(\varphi), \varphi \in \mathcal D(R_+\times R)\}$  with the
covariance structure given by
 \begin{align}\label{CovStru}
   \EE\lk W(\varphi)W(\psi)\rk=\int_{\RR_{+}\times \RR} \cF \varphi(t,\xi) \overline{\cF \psi(t,\xi)}\cdot \mu(d\xi) dt
 \end{align}
for any $\varphi,\psi\in\mathcal{D}(\RR_+\times\RR)$ with $\mu$ given by \eqref{e.c1}.
Notice that \eqref{CovStru} defines a Hilbert scalar product  on $\mathcal{D}(\RR_+\times\RR)$. Denote $\HHH$  the  Hilbert space obtained by completing $\mathcal{D}(\RR_+\times\RR)$ with respect to this scalar product,  which can be expressed in terms of fractional derivatives as following.
\begin{proposition}(\cite[Proposition 2.1]{HW2019}, \cite[Theorem 3.1]{PT2000})\label{hSpaceProp}
 For any $\varphi,\psi\in \mathcal{D}(\RR_+\times\RR)$, the  space $\HHH$ is a Hilbert space equipped with the scalar  product
\begin{equation}
\begin{split}
     \langle\varphi,\psi\rangle_{\HHH}
     :=&\,c_{1,H}\int_{\RR_{+}}\left(\int_ {\RR} \cF \varphi(t,\xi) \overline{\cF \psi(t,\xi)}\cdot |\xi|^{1-2H}d\xi \right)dt
     \\
     =&\,c_{2,H} \int_{\RR_+}\left(\int_{\RR^2}[\varphi(t, x+y)-\varphi(x)]\cdot[\psi(t, x+y)-\psi(x)]\cdot |y|^{2H-2} dxdy\right)dt,
   \end{split}
   \end{equation}
   where $c_{1,H}$ is defined by \eqref{e.c1} and
   \begin{align}
   c_{2,H} :=H^{\frac{1}{2}}\left(\frac 12-H\right)^{\frac{1}{2}} \lk\Gamma\Blc H+\frac 12\Brc\rk^{-1}\left(\int_{0}^{\infty} \Blk(1+t)^{H-\frac 12}-t^{H-\frac 12}\Brk^2 dt+\frac{1}{2H}\right)^{\frac{1}{2}}.  \label{e.c3}
   \end{align}
 \end{proposition}
One can check  that $\varphi(t,x)={\bf 1}_{[0,t]\times[0,x]}$, $t\in\RR_+$, $x\in\RR$, is in $\HHH$ (we set ${\bf 1}_{[0,t]\times[0,x]}=-{\bf 1}_{[0,t]\times[x,0]}$ if $x\in\RR $ is negative ), where $\textbf{1}$ denotes the indicator function. We denote $W(t,x)=W({\bf 1}_{[0,t]\times[0,x]})$. For any $t\geq0$, $\mathcal{F}_t=\sigma\left(W(s,x),0\leq s\leq t,x\in\RR\right)$ be the $\sigma$-algebra generated by  $W$.

The stochastic integral with respect to $W$ is first defined for elementary integrands and then can be extended to general ones.
\begin{definition}\label{ElemP}
An elementary process $g$ is a process given by
   \[
    g(t,x)=\sum_{i=1}^{n}\sum_{j=1}^{m} X_{i,j}\1_{(a_i,b_i]}(t)\1_{(h_j,l_j]}(x),
   \]
   where $n$ and $m$ are finite positive integers, $0\leq a_1<b_1<\cdots<a_n<b_n<\infty$, $h_j<l_j$ and $X_{i,j}$ are $\cF_{a_i}$-measurable random variables for $i=1,\dots,n$, $j=1,\dots,m$. The stochastic integral of   an elementary  process with respect to $W$ is defined as
   \begin{equation}\label{SI_Ele}
     \begin{split}
      \int_{\RR_{+}}\int_{\RR} g(t,x)W(dt,dx)
     =&\, \sum_{i=1}^{n}\sum_{j=1}^{m}  X_{i,j}W(\1_{(a_i,b_i]}\otimes \1_{(h_j,l_j]}) \\
      =&\, \sum_{i=1}^{n}\sum_{j=1}^{m}  X_{i,j}\Big[W(b_i,l_j)-W(a_i,l_j)-W(b_i,h_j)+W(a_i,h_j)\Big].
     \end{split}
   \end{equation}
 \end{definition}

Hu et al. \cite[Proposition 2.3]{HHLNT2017}  extented the the notion of integral with respect to $W$ to a broad class of adapted processes in the following way.
\begin{proposition}(\cite[Proposition 2.3]{HHLNT2017})\label{prop 2.3}
   Let $\Lambda_{H}$ be the space of predictable processes $g$ defined on $\RR_{+}\times\RR$ such that almost surely $g\in\HHH$ and $\EE[\|g\|_{\HHH}^2]<\infty$.  Then, we have that:
   \begin{itemize}
       \item[(i).]
   the space of the elementary processes
    defined by Definition \ref{ElemP} is dense in $\Lambda_{H}$;
      \item[(ii).]
 for any  $g\in\Lambda_{H}$, the stochastic integral $\int_{\RR_{+}}\int_{\RR} g(t,x)W(dt,dx)$ is defined  as the $L^2(\Omega)$-limit of  stochastic integrals of elementary processes approximating $g(t,x)$
in $\Lambda_H$, and for this stochastic integral we have the following isometry equality
    \begin{equation}\label{Isometry}
     \EE\lk\lc\int_{\RR_{+}}\int_{\RR} g(t,x)W(dt,dx)\rc^2\rk=\EE\lk\|g\|_{\HHH}^2\rk.
    \end{equation}
\end{itemize}
\end{proposition}

Denote by $\mathcal{D}(\RR)$ the space of real-valued infinitely differential functions on $\RR$ with compact support. Let $\mathcal H$  be
	  the Hilbert space obtained by completing $\mathcal D(\RR) $ with respect to the following scalar  product:
	  \begin{equation}\label{eq H product}
\begin{split}
     \langle\varphi,\psi\rangle_{\mathcal H}
   =&\,c_{1,H}\int_{\RR} \cF\varphi(\xi) \overline{\cF \psi(\xi)}  \cdot |\xi|^{1-2H}d\xi\\
     =&\, c_{2,H}\int_{\RR^2}[\varphi(x+y)-\varphi(x)]\cdot [\psi(x+y)-\psi(x)]\cdot |y|^{2H-2} dxdy, \ \, \forall \varphi,\psi\in \mathcal D(\mathbb R),
        \end{split}
   \end{equation}
   where $c_{1,H}$ and $c_{2,H}$ are defined by \eqref{e.c1} and \eqref{e.c3}, respectively.
 By Proposition \ref{prop 2.3}, for any orthonormal basis $\{e_k\}_{k\ge1}$ of the Hilbert space $\HH$, the family of processes
\begin{equation}\label{eq int 10}
\left\{B_t^k:=\int_0^t\int_{\mathbb{R}}e_k(y)W(ds,dy)\right\}_{k\ge1}
\end{equation}
is a sequence of independent standard Wiener processes and the process
$
B_t:=\sum_{k\ge1}B_t^ke_k
$
is a cylindrical Brownian motion on $\HH$.
It is well-known that (see \cite{DaPrato1992} or \cite{DQ}) for any $\HH$-valued predictable process
 $g\in L^2(\Omega\times[0,T];\HH)$,
we can define the stochastic integral with respect to the cylindrical Wiener process $B$ as follows:					
\begin{equation}\label{eq int2}
\int_0^T g(s)dB_s:=\sum_{k\ge1}\int_0^T \langle g(s),e_k\rangle_{\HH} dB_s^k.
\end{equation}
  Note that the above series converges in $L^2(\Omega, \FF,\mathbb P)$ and the sum does not depend on
the selected orthonormal basis. Moreover, each summand, in the above series, is a classical It\^o
integral with respect to a standard Brownian motion.

 Let $(B,\| \cdot \|_B)$ be a Banach space  with the norm $\| \cdot \|_B$.  Let   $H\in(\frac{1}{4},\frac{1}{2})$ be a fixed number.
 For  any  function $f:\RR\rightarrow B$,  denote
 \begin{equation}\label{NBNorm}
   \cN_{\frac{1}{2}-H}^{B}f(x):=\lt(\int_{\RR}\|f(x+h)-f(x)\|_B^2\cdot |h|^{2H-2}dh\rt)^{\frac 12},
 \end{equation}
 if the above quantity is finite.
 When $B=\RR$, we abbreviate the notation $\cN_{\frac{1}{2}-H}^{\RR}f$  as  $\cN_{\frac{1}{2}-H}f$.
 As in \cite{HHLNT2017}, when $B=L^p(\Omega)$, we  denote $ \cN_{\frac{1}{2}-H}^{B}$ by $\cN_{\frac{1}{2}-H,\,p}$, that is,
 \begin{equation}\label{NpNorm}
   \cN_{\frac{1}{2}-H,\,p}f( x):=\lt(\int_{\RR}\|f(x+h)-f(x)\|^2_{L^p(\Omega)} \cdot |h|^{2H-2}dh\rt)^{\frac 12}.
 \end{equation}

The following    Burkholder-Davis-Gundy's     inequality is well-known
(see e.g.,  \cite{HHLNT2017,HW2019,LHW2022}).

 \begin{proposition}\label{HBDG}(\cite[Proposition 3.2]{HHLNT2017})
   Let $W$ be the Gaussian noise with the covariance \eqref{CovStru}, and let    $f\in\Lambda_H$  be a predictable random field. Then, we have that, for any $p\geq2$,
   \begin{equation}
     \begin{split}
        \lt\|\int_{0}^{t}\int_{\RR}f(s,y) W(ds,dy)\rt\|_{L^p(\Omega)}
     \leq \sqrt{4p}c_{H}\lt(\int_{0}^{t}\int_{\RR}\lk\cN_{\frac 12-H,\,p}f(s,y)\rk^2dyds\rt)^{\frac 12},\label{e.bdg}
     \end{split}
   \end{equation}
   where $c_{H}$ is a constant depending only on $H$ and
   $\cN_{\frac 12-H,\,p}f(s,y)$ denotes the application of $\cN_{\frac 12-H,\,p} $  to the spatial variable $y$.
 \end{proposition}

\subsection{Stochastic wave equation}

Let $\mathcal{C}([0, T]\times\RR)$ be the space of all continuous real-valued functions on $[0,T]\times\mathbb{R}$, equipped with the metric
\begin{equation}
  d_{\mathcal C}(u,v):=\sum_{n=1}^{\infty}\frac{1}{2^n} \max_{0\le t\le T,|x|\leq n}(|u(t,x)-v(t,x)|\wedge 1).\label{e.4.metric}
 \end{equation}
We introduce the solution space $\mathcal{Z}^p(T)$, which is the space of all continuous functions $v:[0, T ]\times \RR\to L^p(\Omega)$ such that the following norm is finite:
 \begin{equation}\label{ZNorm}
 \begin{split}
   \|v\|_{\mathcal{Z}^p(T)}
   :=&\sup_{t\in[0,T]} \|v(t,\cdot)\|_{L^p(\Omega\times\RR)}+\sup_{t\in[0,T]}\cN^*_{\frac 12-H,\,p}v(t),
 \end{split}
 \end{equation}
where
  \begin{equation}\label{LpNormOmegatimesRR}
  \|v(t,\cdot)\|_{L^p(\Omega\times\RR)}:=\lc\int_{\RR} \EE\lk|v(t,x)|^p\rk dx\rc^{\frac 1p}
 \end{equation}
 and
 \begin{equation}\label{N*pNorm}
  \cN^*_{\frac 12-H,\,p}v(t):=\lc\int_{\RR}\|v(t,\cdot)-v(t,\cdot+h)\|^2_{L^p(\Omega\times\RR)}\cdot |h|^{2H-2}dh\rc^{\frac 12}.
 \end{equation}
Denote by $Z^p(T)$ the space of all non-random functions in $\mathcal Z^p(T)$, which is equipped with the following norm:
\begin{equation}\label{YNorm}
\begin{split}
   \|v\|_{ Z^p(T)}
   :=&\,\sup_{t\in[0,T]} \|v(t,\cdot)\|_{L^p(\RR)}+\sup_{t\in[0,T]}\cN_{\frac 12-H,\, p}^*v(t),
 \end{split}
 \end{equation}
 where
\begin{align}\label{Lpnormal_def}
  \|v(t,\cdot)\|_{L^p(\RR)}:=\lc\int_{\RR} |v(t,x)|^p dx\rc^{\frac 1p}
\end{align}
and
 \begin{align}\label{N*pNorm}
   \cN^*_{\frac 12-H,\,p}v(t):=\lc\int_{\RR}\|v(t,\cdot)-v(t,\cdot+h)\|^2_{L^p(\RR)}\cdot |h|^{2H-2}dh\rc^{\frac 12}.
 \end{align}

Let  $G(t,x):=\frac{1}{2}\textbf{1}_{\{|x|<t\}}$, $t\in {\RR}_+$, $x\in\RR$, be  Green's function associated with
Eq.\,\eqref{SWE}. Notice that $G(t,x)$ does not satisfy the semigroup property.

Recall the following   definition of  the solution to Eq.\,\eqref{SWE} from \cite{LHW2022}.
\begin{definition}(\cite[Definition 2.4]{LHW2022})\label{DefMildSol}
   Let $\{u^\e(t,x)\}_{t\in[0,T],x\in\RR}$ be a real-valued adapted  stochastic field such that for all fixed $t\in[0,T]$ and $x\in\RR$, the random field
   $$
   \{G(t-s,x-y)\sigma(s,y,u^\e(s,y))\textbf{1}_{[0,t]}(s),(s,y)\in\RR_+\times\RR\}
   $$
   is integrable with respect to $W$. The stochastic
   process $u^\e$ is called a strong  solution to \eqref{SWE},  if for all $t\in[0,T]$ and $x\in\RR$ we have almost surely
   \begin{equation}\label{MildSol}
   \begin{split}
     u^\e(t,x)
     =&\, I_0(t,x)+\sqrt{\e}\int_{0}^{t}\int_{\RR} G(t-s,x-y)\sigma(s,y,u^\e(s,y))W(ds,dy),
  \end{split}
   \end{equation}
where
  \begin{equation}\label{I_0Definition}
   \begin{split}
   I_0(t,x):=&\,\frac{\partial}{\partial t}G(t)\ast u_0(x)+G(t)\ast v_0(x)\\
   =&\,\frac{1}{2}[u_0(x+t)+u_0(x-t)]+\frac{1}{2}\int^{x+t}_{x-t}v_0(y)dy.
   \end{split}
   \end{equation}
\end{definition}
The following Hypothesis   guarantees that Eq.\,\eqref{SWE} admits a unique solution.
\begin{Hypothesis}\label{Hypoth.1} We assume that the following two items hold:
 \begin{itemize}
\item[(i).] assume $\sigma(t, x, u)\in\cC^{0,1,1}([0, T]\times\RR^2)$ (the space of all continuous functions $\sigma$, with continuous partial derivatives $\sigma'_x$, $\sigma'_u$ and $\sigma''_{xu}$), satisfying $\sigma(t,x,0)=0$ and there exists a constant $C>0$ such that
  \begin{align}
  \sup_{t\in[0,T], \,x\in\RR,\, u\in\RR} |\sigma'_u(t,x,u)| &\leq C\,; \label{DuSigam}
  \end{align}

  \begin{align}
  \sup_{t\in[0,T], \,x\in\RR, \, u\in\RR} |\sigma''_{xu}(t,x,u)| &\leq C\,; \label{DuxSigam}
  \end{align}

  \begin{equation}\label{DuSigamAdd}
   \sup_{t\in[0,T],  \,x\in\RR}\left| \sigma'_u(t,x,u_1)-\sigma'_u(t,x,u_2)\right| \le C  |u_1-u_2| , \,\,\,\forall u_1,u_2\in \RR;
  \end{equation}
\item[(ii).]the initial condition $u_0$ and $v_0$ are $\alpha$-H\"older continuous with $\alpha\in(0,1]$.
Furthermore, we  assume that $I_0(t,x)$ is in $Z^p(T)$  for some $p>\frac{2}{4H-1}$.
\end{itemize}
\end{Hypothesis}

\begin{remark}
Notice that by \eqref{DuSigam}, we get the Lipschitz continuity of $\sigma$ in $u$ (uniformly in $t$ and $x$), i.e.,

\begin{equation}\label{Lipschitzianuniformlycondu}
\sup_{t\in[0,T],\,x\in\RR}\left|\sigma(t, x, u)-\sigma(t, x, v)\right|\leq C|u-v|,\,\,\,\forall u, v\in\RR;
\end{equation}
for some constant $C>0$, which together with the assumption $\sigma(t,x,0)=0$ implies that

\begin{equation}\label{lineargrowthuniformlycond}
      \sup_{t\in[0,T],\,x\in\RR}\left|\sigma(t, x, u)\right|\leq C|u|,\,\,\,\forall u\in\RR.
\end{equation}

\end{remark}

The following theorem  follows from \cite{LHW2022}.
\begin{theorem}(\cite[Theorem 2.6]{LHW2022})\label{ModUniq}
Assume that Hypothesis \ref{Hypoth.1} holds.
 Then Eq.\,\eqref{SWE} admits a unique strong solution in $\mathcal C([0, T]\times\RR)$ almost surely.
\end{theorem}

\section{Freidlin-Wentzell large deviations and statement of the main result}

\subsection{A general criterion for the large deviation principle}\label{A Criteria for Large Deviations}
Let us first recall some standard definitions and results of the large deviation theory. Let $\{X^\e\}_{\e>0} $ be a family
of random variables defined on a probability space $(\Omega,\mathcal{F},\mathbb{P})$ and taking values in a Polish space $E$.
Roughly speaking, the LDP concerns the exponential decay of the probability measures
of certain kinds of extreme or tail events and the rate of such exponential decay is expressed by the rate function.

 \begin{definition}\label{Dfn-Rate function}
     A function $I: E\rightarrow[0,\infty]$ is called a rate function on
       $E$,
       if for each $M<\infty$ the level set $\{y\in E:I(y)\leq M\}$ is a compact subset of $E$.

    \end{definition}

    \begin{definition}(Large deviation principle)  \label{d:LDP}
       Let $I$ be a rate function on $E$. The sequence
       $\{X^\e\}_{\e>0}$
       is said to satisfy  a large deviation principle on $E$ with the rate function $I$, if the following two
       conditions
       hold:
\begin{itemize}
\item[(a).]   for each closed subset $F$ of $E$,
              $$
                \limsup_{\e\rightarrow 0}\e \log\mathbb{P}(X^\e\in F)\leq -\inf_{y\in F}I(y);
              $$

        \item[(b).]  for each open subset $G$ of $E$,
              $$
                \liminf_{\e\rightarrow 0}\e \log\mathbb{P}(X^\e\in G)\geq-\inf_{y\in G}I(y).
              $$
              \end{itemize}
    \end{definition}

Due to the representation formula of \eqref{eq int 10}, our noise  $W$  can be identified as a sequence of independent, standard, real valued Brownian motions. Set $\mathbb{V}=C([0,T];\mathbb R^{\infty})$.
Let $\{\Gamma^{\varepsilon}\}_{\varepsilon >0}$ be a family of measurable maps from $\mathbb V$ to $E$.
We  recall a criterion for  the LDP of the family $X^{\varepsilon}=\Gamma^{\varepsilon}( W)$ as $\varepsilon \rightarrow 0$.

Define the following space of stochastic processes:
\begin{equation}\label{eq: space}
\mathcal L_2:=\left\{\phi: \Omega\times [0,T]\rightarrow \mathcal H \text{  is predictable and  } \int_0^T\|\phi(s)\|_{\mathcal H}^2ds<\infty, \ \ \mathbb{P}\text{-a.s.} \right\}.
\end{equation}
For each $N\geq1$, let
\begin{equation}\label{eq: space SN}
S^N=\left\{g\in L^2([0,T];\mathcal H): \int_0^T\|g(s)\|_{\mathcal{H}}^2ds\le N  \right\},
\end{equation}
where  $L^2([0,T];\mathcal H)$ is the space of square integrable $\mathcal H$-valued functions on $[0,T]$.
Here and in the sequel of this paper, we will always refer to the weak topology on the set $S^N$. Set $\mathbb S=\bigcup_{N\ge1} S^N$, and
$$
\mathcal U^N=\left\{g\in \mathcal L_2: g(\omega)\in S^N, \,\,\mathbb{P}\text{-a.s.} \right\}.
$$
\begin{condition}\label{Aa} There exists a measurable mapping $\Gamma^0:\mathbb V\rightarrow E$ such that the following two items hold:
\begin{itemize}
\item[(a).]  for every $N<+\infty$, the set $K_N=\left\{\Gamma^0\left(\int_0^{\cdot}   g(s)ds\right): g\in S^N\right\}$ is a compact subset of $E$;
\item[(b).] for every $N<+\infty$ and  $\{g^\e\}_{\e>0}\subset \mathcal U^N$ satisfying that  $g^\e$ converges in distribution (as $S^N$-valued random elements) to $g$, $\Gamma^\e\left(W+\frac{1}{\sqrt\e} \int_0^{\cdot}  g^{\e}(s)ds\right)$ converges in distribution to $\Gamma^0\left(\int_0^{\cdot}  g(s)ds\right)$.
 \end{itemize}
  \end{condition}
  Let
   $I:E \rightarrow [0,\infty]$ be defined  by
\beq\label{rate function}
I(\phi):=\inf_{\left\{g\in \mathbb{S};\,\phi=\Gamma^0\left(\int_0^{\cdot}g(s)ds\right)\right\}}  \left\{\frac12\int_0^T\|g(s)\|_{\mathcal{H}}^2ds\right\},\  \phi\in E,
\nneq
where the infimum over an empty set is taken as $+\infty$.

The following result is due to Budhiraja et al. \cite{BDM2008}.
\begin{theorem}(\cite[Theorem 6]{BDM2008})\label{thm MSZ} For  any $\e>0$, let $X^{\e}=\Gamma^\e(W)$ and suppose that Condition \ref{Aa} holds. Then,  the laws of  $\{X^{\e}\}_{\e>0}$ satisfy an LDP with the rate function  $I$ defined by \eqref{rate function}.
\end{theorem}

We now formulate the following sufficient condition established in  \cite{MSZ}  for verifying the assumptions in Condition \ref{Aa} for LDPs, which is convenient in some applications.
  \begin{condition}\label{cond1} There exists a measurable mapping $\Gamma^0:\mathbb V\rightarrow E$ such that the following two items hold:
\begin{itemize}
\item[(a).] for every $N<+\infty$ and any family $\{g^n\}_{n\ge1}\subset S^N$ that converges to some element $g$ in $S^N$ as $n\rightarrow\infty$, $\Gamma^0\left(\int_0^{\cdot} g^n(s)ds\right)$ converges to  $\Gamma^0\left(\int_0^{\cdot}   g(s)ds\right)$ in the space $E$;
   \item[(b).] for every $N<+\infty$, $\{g^\e\}_{\e>0}\subset \mathcal U^N$  and $\delta>0$,
    $$\lim_{\e\rightarrow 0}\mathbb P\big(\rho(Y^{\e}, Z^{\e})>\delta\big)=0, $$
     where  $Y^{\e}=\Gamma^\e\left(W+\frac{1}{\sqrt\e} \int_0^{\cdot}  g^{\e}(s)ds\right), Z^{\e}=\Gamma^0\left(\int_0^{\cdot}  g^{\e}(s)ds\right)$ and $\rho(\cdot, \cdot)$ stands for the metric in the space $E$.
 \end{itemize}
  \end{condition}
\subsection {Statement of the main result}
In this paper, we are concerned with the following one (spatial) dimensional nonlinear SWE:
\begin{equation*}
\left\{\begin{split}
   &  \frac{\partial^2 u^{\e}(t,x)}{\partial t^2}=\frac{\partial^2 u^{\e}(t,x)}{\partial x^2}+\sqrt{\e}\sigma(t,x,u^{\e}(t,x))\dot{W}(t,x),\quad   t>0,\,
   x\in\RR,\\
   &u^{\e}(0,\cdot)=u_0(x),
  \quad\frac{\partial u^{\e}(0,x)}{\partial t}=v_0(x).
 \end{split}\right.
\end{equation*}
Under Hypothesis \ref{Hypoth.1}, by Theorem \ref{ModUniq}, there exists a unique solution $u^{\e}\in \mathcal{C}([0,T]\times\mathbb{R})$ a.s..  Therefore, there exists a Borel-measurable function $\Gamma^\e: C([0,T];\mathbb R^{\infty}) \to \mathcal{C}([0,T]\times\mathbb{R})$
such that
\begin{equation}\label{eq Gamma e}
u^{\e}(\cdot)=\Gamma^\e(W(\cdot)).
\end{equation}

For every $g\in\mathbb{S}$,  we consider the following deterministic integral equation (the skeleton equation)
\begin{equation*}
u^g(t,x)=I_0(t,x)+\int_0^t\langle G(t-s,x-\cdot)\sigma(s, \cdot, u^g(s, \cdot)), g(s,\cdot)\rangle_{\mathcal H}ds, \ \ \ t\ge0,\,\, x\in \mathbb R.
 \end{equation*}
 The solution $u^g$, whose existence will be proved in the next section, defines a measurable mapping $\Gamma^0:C([0,T];\mathbb R^{\infty}) \to \mathcal{C}([0,T]\times\mathbb{R}) $ so that
\begin{equation}\label{eq Gamma0}
u^g(\cdot)=\Gamma^0\left(\int_0^\cdot g(s)ds\right).
\end{equation}

Here is the main result of this paper.
\begin{theorem}\label{thm LDP}  Assume that Hypothesis \ref{Hypoth.1} holds.
Then, the family $\{u^{\e}\}_{\e>0}$ in Eq.\,\eqref{SWE} satisfies an LDP in the space  $\cC([0, T]\times\RR)$ with the rate function $I$ given by
\begin{equation}\label{eq rate}
I(\phi):=\inf_{\left\{g\in \mathbb S;\, \phi=\Gamma^0\left(\int_0^{\cdot}  g(s)ds\right)\right\} }\left\{ \frac12 \int_0^{T}\|g(s)\|_{\mathcal H}^2ds \right\}.
\end{equation}
\end{theorem}
 \begin{proof}
According to Theorem \ref{thm MSZ}, we only need to prove that  Condition \ref{cond1} is fulfilled. The verification of Condition \ref{cond1} (a) will be given by  Proposition \ref{thm continuity skeleton}. Condition \ref{cond1} (b) will be established in
 Proposition \ref{proposition-4-3}.
The proof is complete.
\end{proof}

\section{Skeleton equation}
In this section, we study the well-posedness of the skeleton equation:
\begin{equation}\label{eq skeleton}
u^g(t,x)=I_0(t,x)+\int_0^t\langle G(t-s,x-\cdot)\sigma(s, \cdot, u^g(s, \cdot)), g(s,\cdot)\rangle_{\mathcal H}ds, \ \ \ t\ge0,\,\, x\in \mathbb R,
 \end{equation}
where $g\in\mathbb{S}$ and $I_0(t,x)$ is defined by \eqref {I_0Definition}.

For any $p\geq 2$, $H\in(\frac 14,\frac 12)$, recall the space $Z^p(T)$ and its norm defined by \eqref{YNorm}.
We have the following well-posedness result for the skeleton equation \eqref{eq skeleton}.
\begin{proposition}\label{thm solu skeleton}
Assume that Hypothesis \ref{Hypoth.1} holds.
  Then, Eq.\,\eqref{eq skeleton} admits a unique solution in $\mathcal{C}([0, T]\times\mathbb R)$. In addition, $\sup_{g\in S^N}\|u^g\|_{Z^p(T)}<\infty$ for any $N\geq 1$ and $p\geq 2$.
  \end{proposition}
 Due to the complexity of the space $\mathcal H$, it is difficult to use a Picard iteration scheme to prove the existence of the solution of Eq.\,\eqref{eq skeleton} directly. We use the approximation method by introducing a new Hilbert space $\mathcal H_{\e}$ as follows.

 For every fixed $\varepsilon>0$, let
\begin{equation}\label{eq f e}
f_{\e}(x):=\mathcal F^{-1}\left(e^{-\e |\xi|^2}|\xi|^{1-2H} \right)(x)=\frac{1}{2\pi} \int_{\mathbb R} e^{i\xi x}e^{-\e |\xi|^2}|\xi|^{1-2H}  d\xi.
	\end{equation}
For any $\varphi, \psi \in \mathcal{D}(\RR)$, we define
	\begin{equation}\label{eq product H e}
    \begin{split}
     \langle\varphi,\psi\rangle_{\mathcal H_{\e}}
     :=&\,c_{1,H}\int_{\RR} \cF \varphi(\xi) \overline{\cF \psi(\xi)} e^{-\e|\xi|^2} |\xi|^{1-2H}d\xi\\
     =&\,c_{1, H}\int_{\mathbb{R}^2}\varphi(x)\psi(y)f_{\e}(x-y)dxdy,
   \end{split}
   \end{equation}
   where  $c_{1,H}$ is given by  \eqref{e.c1}.    Let $\mathcal H_{\e}$  be
	  the Hilbert space obtained by completing $\mathcal D(\RR) $ with respect to the scalar  product given by  \eqref{eq product H e}.
	 Notice that for any $0\le \e_1<\e_2$, we have  that for any $\varphi\in \mathcal H_{\e_1}$,
	 \begin{align}\label{eq H compare}
  \|\varphi\|_{\mathcal H_{\e_1}}\ge \|\varphi\|_{\mathcal H_{\e_2}},
	 \end{align}
and for any $\varphi, \psi\in \mathcal H$, by the dominated convergence theorem,
$$
\lim_{\e\rightarrow0}\, \langle \varphi, \psi \rangle_{\HH_{\e}}=\langle \varphi, \psi \rangle_{\HH}.
$$

For any $g\in \mathbb{S}$, let
\begin{align}\label{eq u h e}
		u_{\ep}^{g}(t,x)=I_0(t,x) + \int_0^t \langle G(t-s,x-\cdot)\sigma(s,\cdot, u^g_{\e}(s,\cdot)), g(s,\cdot) \rangle_{\HH_\e}ds.
\end{align}
 Since $|\xi|^{1-2H}e^{-\varepsilon |\xi|^2}$ is in $L^1(\mathbb R)$, $|f_{\varepsilon}|$ is bounded.  Thus, using the Picard iteration, the existence and uniqueness of the solution $u^{g}_{\e}$ to Eq.\,\eqref{eq u h e} can be proved in a similar (but easier) way to that for  Eq.\,(3.24) in \cite{LHW2022}.

The following lemma asserts that the approximate solution $u_\e^g$  is uniformly bounded in the space $Z^p(T)$ with respect to $\e>0$.
 \begin{lemma}\label{UniBExist}

Let $H\in(\frac 14,\frac 12)$ and $g\in \mathbb{S}$.
Assume that Hypothesis \ref{Hypoth.1} holds.
Then the approximate solution $u^g_\e$ satisfies that, for any $p\ge2$,
   \begin{equation}\label{ReguBdd}
     \sup_{g\in S^N}\sup_{\e>0}\|u_\e^g\|_{Z^p(T)}<\infty.
   \end{equation}
 \end{lemma}
	  \begin{proof}
This proof is inspired by the proof of \cite[Lemma 3.4]{LHW2022}.
We define the Picard iteration sequence as follows.  For $n=0, 1, 2, \cdots$, let
 \begin{align}\label{4-38}
    u_\ep^{g, n+1}(t,x)=&I_0(t,x)+\int_{0}^{t}\langle G(t-s,x-\cdot)\sigma(s,\cdot,u_\ep^{g, n}(s,\cdot)), g(s,\cdot) \rangle_{\HH_\e}ds,
  \end{align}
with $u_\ep^{g, 0}(t,x)=I_0(t,x)$.
Assume that  there exists some constant $N>0$ such that $\int_0^T \|g(s,\cdot)\|_{\HH}^2ds\le N$. Then, by \eqref{eq H compare}, we know that
 \begin{align}\label{g-interal-bound}
 \int_0^T \|g(s,\cdot)\|_{\HH_{\e}}^2ds\le N,\,\,\,\text{for}\,\,\text{any}\,\,\e>0.
 \end{align}
The rest of the proof is divided into four steps. In Step 1, we prove  the convergence of $u_\e^{g, n}(t,\cdot)$ in $L^p(\mathbb R)$ for any $t\in[0,T]$. In Steps 2 and 3,  we  will bound  $\|u_\e^{g,n}(t,\cdot)\|_{L^p(\RR)}$  and $\mathcal N^*_{\frac12-H,\, p} u_{\e}^{g,n}(t)$ for each fixed $\ep>0$. Step 4 is devoted to  proving that  $u_\e^g$ is   bounded  in  $\left(Z^p(T),\, \|\cdot\|_{Z^p(T)}\right)$ uniformly over    $\e>0$.

{\bf Step 1.}  We will  bound $\|u_\e^{g, n}(t,\cdot)\|_{L^p(\RR)}$ uniformly in $n$  and show that $u_{\e}^{g}(t,\cdot)$ is in $L^p( \RR)$ in this step.  By Cauchy-Schwarz's inequality, \eqref{Lipschitzianuniformlycondu}, \eqref{g-interal-bound}, the boundedness of $f_ {\e}$ and Jensen's inequality with respect to $\frac{1}{t-s}G(t-s,x-y)dy$, we have that  for any $t\in [0,T]$ and $x\in \RR$,
     \begin{align*}
      & |u_\e^{g, n+1}(t,x)-u_\e^{g,n}(t,x)|^2 \\
     =&\, \left|\int_{0}^{t}\left\langle G(t-s,x-\cdot)\big[\sigma(s,\cdot,u_\ep^{g, n}(s,\cdot))-\sigma(s,\cdot,u_\ep^{g, n-1}(s,\cdot))\big], g(s,\cdot)\right\rangle_{\HH_\e}ds\right|^2  \\
     \le  &\,   \int_0^t \|g(s,\cdot)\|_{\HH_\e}^2ds\cdot \int_{0}^{t}\left\|G(t-s,x-\cdot)\left[\sigma(s,\cdot,u_\ep^{g, n}(s,\cdot))-\sigma(s,\cdot,u_\ep^{g, n-1}(s,\cdot))\right]\right\|_{\HH_\e}^2ds\\
     \leq&\, c_{1,H} N    \int_0^t \int_{\RR^2}  G(t-s,x-y)\left[\sigma(s,y,u_\ep^{g, n}(s,y))-\sigma(s,y,u_\ep^{g, n-1}(s,y))\right] \\
     &\qquad\qquad\qquad\quad \cdot G(t-s,x-z)\left[\sigma(s,z,u_\ep^{g, n}(s,z))-\sigma(s,z,u_\ep^{g, n-1}(s,z))\right]f_{\e}(y-z)dydzds\\
     \le&\,c_{1,H}  TN\|f_{\e}\|_{\infty} \cdot\int_0^t\int_{\RR } G(t-s,x-y)\left[\sigma(s,y,u_\ep^{g, n}(s,y))-\sigma(s,y,u_\ep^{g, n-1}(s,y))\right]^2dyds\\
    \le&\,C_{\e, T,N,H}\int_0^t\int_{\RR } G(t-s,x-y)\left|u_\ep^{g, n}(s,y)-u_\ep^{g, n-1}(s,y)\right|^2dyds .
   \end{align*}
By Jensen's inequality with respect to $\frac{1}{t(t-s)}G(t-s,x-y)dyds$ and  a change of variable,   we have that for any $p\ge2$,  $t\in [0,T]$,
  \begin{align}\label{Phi-g-n-l-estimate}
  & \|u_\e^{g, n+1}(t,\cdot)-u_\e^{g, n}(t,\cdot)\|_{L^p(\RR)}^p\notag\\
     \leq&\,C_{\ep,T,N,p,H}\int_{\RR}\left[\int_{0}^{t}\int_{\RR}G(t-s,x-y)\left|u_\ep^{g, n}(s,y)-u_\ep^{g, n-1}(s,y)\right|^2dyds\right]^{\frac{p}{2}}dx\notag\\
      \leq&\,C_{\ep,T,N,p,H}\int_{0}^{t}\int_{\RR}G(t-s,x)dx\cdot\int_{\RR}\left|u_\ep^{g, n}(s,y)-u_\ep^{g, n-1}(s,y)\right|^pdyds\\
       \leq&\,C_{\ep,T,N,p,H}\int_{0}^{t}\left\|u_\ep^{g, n}(s,\cdot)-u_\ep^{g, n-1}(s,\cdot)\right\|^p_{L^p( \RR)}ds\notag\\
     \le &\,C_{\ep,T,N,p,H}^n  \frac{T^n}{n!} \sup_{s\in[0,T]}\|u_\ep^{g, 1}(s,\cdot)-u_\ep^{g, 0}(s,\cdot)\|_{L^p( \RR)}^p.\notag
   \end{align}
 Thus, \eqref{Phi-g-n-l-estimate} implies that
 $$
   \sup_{n\geq1} \sup\limits_{t\in[0,T]}\|u_{\ep}^{g, n}(t,\cdot)\|_ {L^p(\RR)}<\infty, \quad \hbox{ for each $\ep>0$},
  $$
  and that $\{u_{\ep}^{g, n}(t,\cdot)\}_{n\geq1}$ is a Cauchy sequence in $L^p( \RR)$ for any $t\in[0,T]$. Hence, for any fixed $t\in[0,T]$, there exists $u_{\e}^{g}(t,\cdot)\in L^p( \RR)$ such that $u_{\e}^{g, n}(t,\cdot)$ converges to $u_{\e}^{g}(t,\cdot)$ in $L^p( \RR)$ when $n$ goes to infinity.

\textbf{Step 2.} In this step, we will provide a uniform bound of $\|u_\e^{g,n}(t,\cdot)\|_{L^p(\RR)}$ for any fixed $\e>0$. For the simplicity of writing, denote that
\begin{align}\label{TechDt}
\mathcal{D}_hf(t,x):=f(t,x+h)-f(t,x)
\end{align}
and
\begin{align}\label{TechBoxt}
\Box_{h,l}f(t,x):=f(t,x+h+l)-f(t,x+l)-f(t,x+h)+f(t,x),
\end{align}
for any function $f$.

By the Cauchy-Schwarz inequality, \eqref{eq H product}, \eqref{eq H compare} and \eqref{g-interal-bound}, we have
   \begin{align}\label{uepLp}
      |u_\e^{g,n+1}(t,x)|^p
      \lesssim\,&\left|I_0(t,x)\right|^p+\left|\int_{0}^{t}\left\langle G(t-s,x-\cdot)\sigma(s,\cdot,u_\e^{g,n}(s,\cdot)), g(s,\cdot)\right\rangle_{\HH_\e}ds\right|^p\notag\\
     \lesssim\,&\left|I_0(t,x)\right|^p+ \lc\int_{0}^{t}\left\| G(t-s,x-\cdot)\sigma(s,\cdot,u_\e^{g,n}(s,\cdot))\right\|_{\HH_\e}^2ds\rc^{\frac p2}\notag\\
        \lesssim\,&\left|I_0(t,x)\right|^p+\lc\int_{0}^{t}\left\| G(t-s,x-\cdot)\sigma(s,\cdot,u_\e^{g,n}(s,\cdot))\right\|_{\HH}^2ds\rc^{\frac p2} \\
     \simeq\,&\left|I_0(t,x)\right|^p+ \bigg(\int_{0}^{t}\int_{\RR^2}  \Big|G(t-s,x-y-z)\sigma(s,y+z,u_\e^{g,n}(s,y+z))\notag\\
       &\qquad\qquad\qquad-G(t-s,x-y)\sigma(s,y,u_\e^{g,n}(s,y))\Big|^2 \cdot|z|^{2H-2}dydzds\bigg)^{\frac p2}\notag\\
       \lesssim\,& \left|I_0(t,x)\right|^p+A_1(t,x)+A_2(t,x)+A_3(t,x),\notag
   \end{align}
where
   \begin{align*}
     A_1(t,x) :=\,&\bigg(\int_{0}^{t}\int_{\RR^2} |G(t-s,x-y-z)|^2 \cdot  \big|\sigma(s,y+z,u_\e^{g,n}(s,y+z))
    -\sigma(s,y,u_\e^{g,n}(s,y+z))\big|^2\\ &\qquad\quad\cdot|z|^{2H-2}dzdyds\bigg)^{\frac p2} ;\\
     A_2(t,x) :=\,&\bigg(\int_{0}^{t}\int_{\RR^2}  |G(t-s,x-y-z)|^2 \cdot  \big|\sigma(s,y,u_\e^{g,n}(s,y+z))-\sigma(s,y,u_\e^{g,n}(s,y))\big|^2\\
     &\qquad\quad\cdot|z|^{2H-2}dzdyds\bigg)^{\frac p2};\\
     A_3(t,x) :=\,&\bigg( \int_{0}^{t}\int_{\RR^2}|\mathcal{D}_{z}G(t-s,x-y)|^2\cdot\left|\sigma(s,y,u_\e^{g,n}(s,y))\right|^2
             \cdot|z|^{2H-2} dzdyds\bigg)^{\frac p2}.
   \end{align*}
   If $|z|>1$, then we have that by   \eqref{lineargrowthuniformlycond},
   \begin{equation}\label{DuSigamy}
  \begin{split}
   &\left|\sigma(s,y+z,u_\e^{g,n}(s,y))-\sigma(s,y,u_\e^{g,n}(s,y)) \right|^2\\
   \lesssim\,&\left|\sigma(s,y+z,u_\e^{g,n}(s,y)) \right|^2+\left|\sigma(s,y,u_\e^{g,n}(s,y)) \right|^2\\
    \lesssim\,&\left|u_\e^{g,n}(s,y) \right|^2.
      \end{split}
     \end{equation}
If $|z|\leq1$, due to $\sigma(t,x,0)=0$ and \eqref{DuxSigam}, we have
   \begin{equation}\label{DuxSigamy}
   \begin{split}
   \left|\sigma(s,y+z,u_\e^{g,n}(s,y))-\sigma(s,y,u_\e^{g,n}(s,y)) \right|^2
   =\,&\left|\int_0^{u_\e^{g,n}}[\sigma'_{\xi}(s,y+z,\xi)-\sigma'_{\xi}(s,y,\xi)d\xi \right|^2\\
    \lesssim\,& \left|u_\e^{g,n}(s,y) \right|^2\cdot |z|^2.
     \end{split}
     \end{equation}
Since $H\in\left(\frac{1}{4},\frac{1}{2}\right)$, by a change of variable, \eqref{DuSigamy}, \eqref{DuxSigamy} and Minkowski's inequality, we have
   \begin{equation}\label{D1Bdd}
   \begin{split}
 \lc\int_\RR A_1(t,x)dx \rc^{\frac 2p}
\lesssim\,&\lc\int_\RR \bigg(\int_{0}^{t}\int_{\RR}|G(t-s,y)|^2\cdot  \big|u_\e^{g,n}(s,x)\big|^2dyds\bigg)^{\frac p2}dx \rc^{\frac 2p}\\
 \lesssim\,&
 \int_{0}^{t}\int_{\RR}|G(t-s,y)|^2dy\cdot \left(\int_{\RR}\big|u_\e^{g,n}(s,x)\big|^pdx \right)^{\frac 2p}ds\\
 \lesssim\,&\int_{0}^{t} (t-s)\cdot \|u_\e^{g,n}(s,\cdot)\|^2_{L^p(\RR)}ds.
   \end{split}
   \end{equation}
   By a change of variable, \eqref{Lipschitzianuniformlycondu} and  Minkowski's inequality, we have
   \begin{equation}\label{D1Bdd2}
   \begin{split}
 &\lc\int_\RR A_2(t,x)dx \rc^{\frac 2p}\\
 \lesssim\,&\lc\int_\RR \bigg(\int_{0}^{t}\int_{\RR^2}  |G(t-s,y)|^2\cdot  \big|u_\e^{g,n}(s,x+z)-u_\e^{g,n}(s,x)\big|^2\cdot|z|^{2H-2}dzdyds\bigg)^{\frac p2}dx \rc^{\frac 2p}\\
 \lesssim\,&\int_{0}^{t}\int_{\RR}|G(t-s,y)|^2dy\cdot \int_{\RR}\left(\int_{\RR}\big|u_\e^{g,n}(s,x+z)-u_\e^{g,n}(s,x)\big|^pdx \right)^{\frac 2p}\cdot|z|^{2H-2}dzds\\
 \lesssim \,&\int_{0}^{t}(t-s)\cdot \lt[\cN^*_{\frac 12-H,\,p}u_\e^{g,n}(s)\rt]^2 ds.
   \end{split}
   \end{equation}
   By \eqref{lineargrowthuniformlycond}, a change of variable, Minkowski's inequality and Lemma \ref{TechLemma3}, 
   we have
   \begin{align}\label{D1Bdd3}
\lc\int_\RR A_3(t,x)dx \rc^{\frac 2p}
 \lesssim\,&\lc\int_\RR \bigg( \int_{0}^{t}\int_{\RR^2}|\mathcal{D}_{z}G(t-s,y)|^2\cdot\left|u_\e^{g,n}(s,x)\right|^2
             \cdot|z|^{2H-2} dzdyds\bigg)^{\frac p2}dx \rc^{\frac 2p}\notag \\
 \lesssim\,&\int_{0}^{t}\int_{\RR^2}|\mathcal{D}_{z}G(t-s,y)|^2|z|^{2H-2}dzdy\cdot \left(\int_{\RR}\big|u_\e^{g,n}(s,x)\big|^pdx \right)^{\frac 2p}ds\\
 \lesssim\,&\int_{0}^{t}(t-s)^{2H}\cdot \|u_\e^{g,n}(s,\cdot)\|^2_{L^p( \RR)}ds.\notag
   \end{align}
Thus, putting \eqref{uepLp}, \eqref{D1Bdd}, \eqref{D1Bdd2} and  \eqref{D1Bdd3} together, we have
\begin{equation}\label{ueptcdot-350}
\begin{split}
  \|u_\e^{g,n+1}(t,\cdot)\|_{L^p( \RR)}^2
 \ls\,& \left\|I_0(t,\cdot)\right\|^2_{L^p(\RR)}+ \int_{0}^{t} \left((t-s)+(t-s)^{2H}\right)\cdot \|u_\e^{g,n}(s,\cdot)\|^2_{L^p( \RR)}  ds \\
 &+\int_{0}^{t}(t-s) \lt[\cN^*_{\frac 12-H,\,p}u_\e^{g,n}(s)\rt]^2  ds.
\end{split}
\end{equation}

{\bf Step 3.} This step is devoted to estimating $\mathcal N^*_{\frac12-H,\, p} u_\e^{g,n}(t)$ for any fixed $\e>0$.
Using the Cauchy-Schwarz inequality, \eqref{eq H compare} and \eqref{g-interal-bound}, we have
   \begin{align*}
      & \left|u_\e^{g,n+1}(t,x+h)-u_\e^{g,n+1}(t,x)\right|^p \\
     \lesssim\,&
      \left|I_0(t,x+h)-I_0(t,x)\right|^p+\lc\int_{0}^{t}\left\| \mathcal{D}_hG(t-s,x-\cdot)\sigma(s, \cdot, u_\e^{g,n}(s,\cdot))\right\|_{\HH_\e}^2ds\rc^{\frac p2}  \nonumber\\
     \lesssim\,& \left|I_0(t,x+h)-I_0(t,x)\right|^p+
       \bigg( \int_{0}^{t}\int_{\RR^2} \Big| \mathcal{D}_hG(t-s,x-y-z)\sigma(s, y+z, u_\e^{g,n}(s,y+z)) \\
       &\qquad\qquad \qquad\qquad \qquad\qquad- \mathcal{D}_hG(t-s,x-y)\sigma(s, y, u_\e^{g,n}(s,y))\Big|^2\cdot|z|^{2H-2} dydzds\bigg)^{\frac p2} \\
     \lesssim\,&I_0(t,x,h)+ I_1(t,x,h)+I_2(t,x,h)+I_3(t,x,h),
   \end{align*}
   where
   \begin{align*}
   I_0(t,x,h):=\,& \left|I_0(t,x+h)-I_0(t,x)\right|^p;\\
     I_1(t,x,h):=\,&  \bigg(\int_{0}^{t}\int_{\RR^2}\left| \mathcal{D}_hG(t-s,x-y-z)\right|^2 \\
     &\quad \qquad\cdot\big|\sigma(s,y+z,u_\e^{g,n}(s,y+z))-\sigma(s, y, u_\e^{g,n}(s,y+z)) \big|^2\cdot |z|^{2H-2}dydzds\bigg)^{\frac p2};\\
     I_2(t,x,h):=\,&  \bigg(\int_{0}^{t}\int_{\RR^2}\left| \mathcal{D}_hG(t-s,x-y-z)\right|^2 \\
     &\quad\qquad\cdot\big|\sigma(s,y,u_\e^{g,n}(s,y+z))-\sigma(s,y,u_\e^{g,n}(s,y)) \big|^2\cdot |z|^{2H-2}dydzds\bigg)^{\frac p2};\\
     I_3(t,x,h):=\,& \bigg(\int_{0}^{t}\int_{\RR^2} \left|\Box_{h,z}(t-s,x-y)\right|^2\cdot  \big|\sigma(s,y,u_\e^{g,n}(s,y))\big|^2\cdot |z|^{2H-2} dzdyds\bigg)^{\frac p2}.
   \end{align*}
   Therefore, by \eqref{N*pNorm}, we have
   \begin{equation}\label{sum-I-1-3-445}
   \begin{split}
     \lk\cN^*_{\frac 12-H,\,p}u_\e^{g,n+1}(t)\rk^2
     \lesssim&\,\sum_{j=0}^{3}\int_{\RR}\lc \int_\RR I_j(t,x,h) dx\rc^{\frac 2p} \cdot |h|^{2H-2} dh.
   \end{split}
   \end{equation}
By \eqref{DuSigamy}, \eqref{DuxSigamy}, a change of variable, Minkowski's inequality and  Lemma \ref{TechLemma3},  we have
   \begin{equation}\label{I1.Bdd}
   \begin{split}
      &\int_{\RR}\lt| \int_\RR I_1(t,x,h)dx\rt|^{\frac 2p}\cdot |h|^{2H-2} dh\\
        \lesssim\,&\int_{\RR}\bigg[ \int_\RR \bigg(\int_{0}^{t}\int_{\RR}\left| \mathcal{D}_hG(t-s,y)\right|^2 \big|u_\e^{g,n}(s,x) \big|^2 dyds\bigg)^{\frac p2} dx\bigg]^{\frac 2p} \cdot|h|^{2H-2} dh \\
      \lesssim\,&\int_{0}^{t}\int_{\RR^2}\left|  \mathcal{D}_hG(t-s,y)\right|^2\cdot |h|^{2H-2}dhdy\left(\int_\RR\big|u_\e^{g,n}(s,x) \big|^p  dx\right)^{\frac{2}{p}}ds\\
      \lesssim\,& \int_{0}^{t} (t-s)^{2H}\cdot\left\|u_\e^{g,n}(s,\cdot)\right\|_{L^p( \RR)}^{2} ds.
   \end{split}
   \end{equation}
By \eqref{Lipschitzianuniformlycondu},  a change of variable, Minkowski's inequality, Jensen's inequality and Lemma \ref{TechLemma3}, we have
   \begin{align}\label{I2.Bdd}
      &\int_{\RR}\lt| \int_\RR I_2(t,x,h)dx\rt|^{\frac 2p} \cdot|h|^{2H-2} dh \notag\\
       \lesssim\,&\int_{\RR}\bigg[ \int_\RR \bigg(\int_{0}^{t}\int_{\RR^2}\left| \mathcal{D}_hG(t-s,y)\right|^2\cdot \big|u_\e^{g,n}(s,x+z)-u_\e^{g,n}(s,x) \big|^2\cdot |z|^{2H-2}dydzds\bigg)^{\frac p2}dx\bigg]^{\frac 2p} \cdot|h|^{2H-2} dh \notag\\
      \lesssim\,&\int_{0}^{t}\int_{\RR^2}\left| \mathcal{D}_hG(t-s,y)\right|^2 \cdot |h|^{2H-2}dhdy\cdot\int_{\RR}\left(\int_\RR\big|u_\e^{g,n}(s,x+z)-u_\e^{g,n}(s,x) \big|^p dx\right)^{\frac{2}{p}}\cdot |z|^{2H-2}
      dzds\notag\\
      \lesssim\,& \int_{0}^{t} (t-s)^{2H}\cdot \lk\cN^*_{\frac 12-H,\,p}u_\e^{g,n}(s)\rk^2 ds. \notag\\
\end{align}
By  \eqref{lineargrowthuniformlycond},  a change of variable, Minkowski's inequality and  Lemma \ref{TechLemma3}, we have
   \begin{equation}\label{I3_bddestimate}
   \begin{split}
      &\int_{\RR}\lt| \int_\RR I_{3}(t,x,h)dx\rt|^{\frac 2p} \cdot|h|^{2H-2} dh\\
     \lesssim\,&\int_{\RR}\bigg[ \int_\RR \bigg(\int_{0}^{t}\int_{\RR^2}\left|\Box_{h,z}(t-s,y)\right|^2\cdot \big|u_\e^{g,n}(s,x) \big|^2\cdot |z|^{2H-2}dydzds\bigg)^{\frac p2}dx\bigg]^{\frac 2p} \cdot|h|^{2H-2} dh \\
      \lesssim\,&\int_{0}^{t}\int_{\RR^3}|\Box_{h,z}(t-s,y)|^2\cdot\left(\int_\RR|u_\e^{g,n}(s,x)|^pdx\right)^{\frac{2}{p}}
      |h|^{2H-2} \cdot |z|^{2H-2}dhdzdyds \\
       \lesssim\,&\int_{0}^{t} (t-s)^{4H-1}\cdot \|u_\e^{g,n}(s,\cdot)\|^2_{L^p(\RR)}ds.
     \end{split}
     \end{equation}
   Thus, substitute  \eqref{I1.Bdd}, \eqref{I2.Bdd} and \eqref{I3_bddestimate} back into \eqref{sum-I-1-3-445}, we have
   \begin{equation}\label{ueNxing1}
   \begin{split}
     \lk\cN^*_{\frac 12-H,\,p}u_\e^{g,n+1}(t)\rk^2
     \lesssim\,& \lk\cN^*_{\frac 12-H,\,p}I_0(t)\rk^2+\int_{0}^{t} \left((t-s)^{2H}+(t-s)^{4H-1}\right)\cdot  \left\|u_\e^{g,n}(s,\cdot)\right\|_{L^p( \RR)}^{2} ds\\
     &+\int_{0}^{t}(t-s)^{2H}\cdot   \lk\cN^*_{\frac 12-H,\,p}u_\e^{g,n}(s)\rk^2 ds.
   \end{split}
   \end{equation}

{\bf Step 4.}
Define
$$
    \Psi_\e^n(t):=\|u_\e^{g,n}(t,\cdot)\|_{L^p(\RR)}^{2}+\lk\cN^*_{\frac 12-H,\,p}u_\e^{g,n}(t)\rk^2.
  $$
Combining  \eqref{ueptcdot-350} and \eqref{ueNxing1},  there exists a constant $C_{T,p,H,N}$ such that
   \begin{align*}
     \Psi_\e^{n+1}(t)\leq C_{T,p,H,N}\left(\|I_0\|^2_{Z^p(T)} +\int_{0}^{t} (t-s)^{4H-1}\Psi_\e^n(s) ds\right).
  \end{align*}
By the extension of Gronwall's lemma \cite[Lemma 15]{Dalang1999}, we have
\begin{align}\label{eqcgbound}
    \sup_{n\geq 1}\sup_{t\in[0,T]}\Psi_\e^n(t)\leq C,
\end{align}
where $C$ is a constant independent of $\e$ and  $g\in S^N$.

For any fixed $\e>0$, since $u_{\e}^{g, n}(t,\cdot)\rightarrow u_{\e}^{g}(t,\cdot)$  in $L^p(\RR) $  as $n\rightarrow\infty$,
we have that by Fatou's lemma,
  $$ \sup_{g\in S^N}\sup_{\e>0}\|u_\e^{g}\|_{Z^p(T)}\leq \sup_{g\in S^N}\sup_{\e>0}\sup_{t\in[0,T]}\|u^{g}_\ep(t,\cdot)\|_{L^p(\RR)}+
 \sup_{g\in S^N}\sup_{\e>0}\sup_{t\in[0,T]}\cN^*_{\frac 12-H,\,p}u_\e^{g}(t)<\infty.$$

  The proof is complete.
\end{proof}

 According to Lemmas \ref{LEMMA3.4}, \ref{D_hJtheKestimate}  and \ref{Holder_continu_u0v0}, by the same argument as that in the proofs of \cite[Proposition 3.4]{LHW2022} and \cite[Proposition 4.1]{LHW2022}, the following lemma can be proved and  we  omit the details here.
 \begin{lemma}\label{TimeSpaceRegBdd}
Let $u_\ep^g$ be the approximate mild solution defined by \eqref{eq u h e}. Then we have the following results:
   \begin{itemize}
     \item[(i).] if {\bf $p>\frac{2}{4H-1}$}, then there exists a constant $C_{T,p,H ,N}>0$ such that
     \begin{equation}\label{359359}
      \sup\limits_{t\in[0,T],\,x\in \RR}  \cN_{\frac 12-H}u_\ep^g(t,x) \leq C_{T,p,H,N}\|u_\ep^g\|_{Z^p(T)};
     \end{equation}
     \item[(ii).] if $p>\frac{1}{H}$ and
     $0<\gamma< H-\frac{1}{p}$, then there exists a constant $C_{T ,p,H,N,\gamma}>0$ such that
     \begin{equation}\label{Hol1}
       \sup\limits_{\substack{t,\,t+h\in[0,T], \\ x\in \RR}} |u_\e^g(t+h,x)-u_\e^g(t,x)| \leq C_{T ,p,H,N,\gamma}|h|^{\gamma}\cdot\|u_\ep^g\|_{Z^p(T)};
     \end{equation}
     \item[(iii).] if {\bf $p>\frac{1}{H}$ }and
   {\bf$0<\gamma<H-\frac{1}{p}$}, then there exists a constant $C_{T ,p,H,N,\gamma}>0$ such that
     \begin{equation}\label{Lemma-4.3-iii}
       \sup\limits_{\substack{t\in[0,T],\,  x,\,y\in \RR}}|u_\e^g(t,x)-u_\e^g(t,y)|\leq C_{T,p,H,N,\gamma}|x-y|^{\gamma}\cdot\|u_\e^g\|_{Z^p(T)}.
     \end{equation}
   \end{itemize}
 \end{lemma}

 \begin{proof}[Proof of Proposition \ref{thm solu skeleton}] {\bf Existence.}
 The uniform H\"older continuity of the type specified in Lemma \ref{TimeSpaceRegBdd} (ii) and (iii) is the key to show the existence of the solution to Eq.\,\eqref{eq skeleton}.
 This, together with the fact of $u^{g}_{\e}(0,0)=u_0(0)$ yields that the family $\{u^{g}_{\e}\}_{\e>0}$ is relatively compact in $({\mathcal C}([0,T]\times\RR), d_{\mathcal C})$ by the Arzel\`a-Ascoli theorem.
Hence, there is a subsequence $\ep_n\downarrow 0$ such that    $u^g_{\ep_n}\rightarrow u^g$ in $(\cC([0, T]\times\RR),d_{\cC})$.
By Lemma  \ref{thm solu skeleton} and Lemma \ref{lemma3.3limits} (i),
 we have  that, for any $p\ge 2$,
\begin{align}\label{ugZlamdanormal}
\sup_{g\in S^N}\|u^g\|_{Z^p(T)}<\infty.
\end{align}
By using the Cauchy-Schwarz inequality, \eqref{Lipschitzianuniformlycondu},  Lemma \ref{lemma3.3limits} (ii) and the dominated convergence theorem, we have that as $\e_n\downarrow 0$,
  \begin{align*}
 u^g_{\e_n}(t,x)=&\,I_0(t,x)+\int_0^t\langle G(t-s,x-\cdot)\left(\sigma(s,\cdot,u^g_{\ep_n}(s,\cdot))-\sigma(s,\cdot, u^g(s,\cdot))\right),g(s,\cdot)\rangle_{\cH_{\ep_n}} ds\notag\\
 &+\int_0^t\langle G(t-s,x-\cdot)\sigma(s,\cdot,u^g(s,\cdot)),g(s,\cdot)\rangle_{\cH_{\ep_n}} ds\notag\\
  &\,\longrightarrow  I_0(t,x)+\int_0^t\langle G(t-s,x-\cdot)(x-\cdot)\sigma(s,\cdot, u^g(s,\cdot)),g(s,\cdot)\rangle_{\cH} ds.
  \end{align*}
 The uniqueness of the limit of  $\{u^g_{\ep_n}\}_{n\ge1}$ implies that $u^g$ satisfies Eq.\,\eqref{eq skeleton}.

{\bf Uniqueness.}
Let $u^g$ and $v^g$ be two solutions of Eq.\,\eqref{eq skeleton} with the same initial value condition.
Similarly to  Lemma \ref{TimeSpaceRegBdd} (i), we can obtain that, for any   $p>\frac{2}{4H-1}$,
\begin{equation}\label{Nug}
      \sup\limits_{t\in[0,T],\, x\in \RR}\cN_{\frac 12-H}u^g(t,x) <\infty,\,\,
      \sup\limits_{t\in[0,T],\, x\in \RR} \cN_{\frac 12-H}v^g(t,x) <\infty.
\end{equation}
  Denote that
  \begin{align*}
  S_1(t)=\,\int_{\RR}|u^g(t,x)-v^g(t,x)|^2dx
  \end{align*}
and
   \begin{align*}
  S_2(t)=\,\int_{\RR^2}\left|u^g(t,x+h)-v^g(t,x+h)-u^g(t,x)+v^g(t,x)\right|^2\cdot|h|^{2H-2}dhdx.
  \end{align*}
According to \eqref{ugZlamdanormal} and the definition of the norm $\|\cdot\|_{Z^p(T)}$ , we know that
  \begin{align}\label{S_1S_2_infty}
  \sup_{t\in[0,T]}S_1(t)<\infty,\qquad\sup_{t\in[0,T]}S_2(t)<\infty.
  \end{align}
 Recall $\mathcal{D}_hf(t,x)$ defined by \eqref{TechDt} and denote that $$\triangle(t,x,y):=\sigma(t,x,u^g(t,y))-\sigma(t,x,v^g(t,y)).$$
  Since $\int_0^T \|g(s,\cdot)\|_{\HH}^2ds\le N$, by using the Cauchy-Schwarz inequality, \eqref{eq H product}  and a change of variable, we have
  \begin{align}\label{4.84}
  &\int_{\RR}|u^g(t,x)-v^g(t,x)|^2dx\notag\\
  =&\,\int_{\RR}\left|\int_0^t\langle G(t-s, x-\cdot)\triangle(s,\cdot,\cdot), g(s,\cdot)\rangle_{\mathcal H}ds\right|^2 dx\notag\\
 \lesssim&\,\int_{\RR}
\int_0^t\|G(t-s,x-\cdot)\triangle(s,\cdot,\cdot)\|^2_{\mathcal H}dsdx\notag\\
 \lesssim&\,\int_0^t\int_{\RR^3}\left|G(t-s,x-y)\right|^2\left|\triangle(s,y,y+h)-\triangle(s,y,y)\right|^2\cdot|h|^{2H-2}dhdydxds\\
  &+\int_0^t\int_{\RR^3}\left|\mathcal{D}_{-h}G(t-s,x-y)\right|^2\left|\triangle(s,y+h,y)\right|^2\cdot|h|^{2H-2}dhdydxds\notag\\
  &+\int_0^t\int_{\RR^3}\left|G(t-s,x-y)\right|^2\left|\triangle(s,y+h,y)-\triangle(s,y,y)\right|^2\cdot|h|^{2H-2}dhdydxds\notag\\
 =:&\,V_1(t)+V_2(t)+V_3(t).\notag
  \end{align}
  According to   \eqref{DuSigam} and (\ref{DuSigamAdd}), we have
  \begin{equation}\label{490-490}
  \begin{split}
 &\left|\triangle(s,y,y+h)-\triangle(s,y,y)\right|^2\\
 =&\,\Bigg|\int_0^1\left[u^g(s,y+h)-v^g(s,y+h)\right]\sigma'_\xi(s,y,\theta u^g(s,y+h)+(1-\theta) v^g(s,y+h))d\theta\\
 &-\int_0^1\left[u^g(s,y)-v^g(s,y)\right]\sigma'_\xi(s,y,\theta u^g(s,y)+(1-\theta) v^g(s,y))d\theta\Bigg|^2\\
 \lesssim&\,\left|u^g(s,y+h)-v^g(s,y+h)-u^g(s,y)+v^g(s,y)\right|^2\\
 &+\left|u^g(s,y)-v^g(s,y)\right|^2\cdot\left[|u^g(s,y+h)-u^g(s,y)|^2+|v^g(s,y+h)-v^g(s,y)|^2\right].
 \end{split}
 \end{equation}
This,  together with \eqref{Nug}, gives the  following estimate:
 \begin{align}\label{v3label}
 V_1(t)\lesssim\, \int^t_0(t-s)\cdot [S_1(s)+S_2(s)]ds.
 \end{align}
 By Lemma \ref{TechLemma3} and \eqref{Lipschitzianuniformlycondu}, we have
 \begin{equation}\label{v1label}
 \begin{split}
 V_2(t)
 \lesssim&\,\int_0^t\int_{\RR}\left(\int_{\RR^2}\left|\mathcal{D}_hG(t-s,x-y)\right|^2|h|^{2H-2}dhdx\right)\cdot|u^g(s,y)-v^g(s,y)|^2dyds\\
  \lesssim&\, \int_0^t(t-s)^{2H}\cdot S_1(s)ds.
   \end{split}
  \end{equation}
If  $h>1$, \eqref{DuSigam} implies that
 \begin{equation}\label{h1dayu}
 \begin{split}
 \left|\triangle(s,y+h,y)-\triangle(s,y,y)\right|^2
 =&\,\left|\int^{v^g}_{u^g}\left(\sigma'_\xi(s,y+h,\xi)-\sigma'_\xi(s,y,\xi)\right)d\xi\right|^2\\
 \lesssim&\, \left|u^g(s,y)-v^g(s,y)\right|^2.
  \end{split}
  \end{equation}
 If  $h\leq1$, then we have that by   \eqref{DuxSigam}
 \begin{equation}\label{h1xiaoyu}
 \begin{split}
 \left|\triangle(s,y+h,y)-\triangle(s,y,y)\right|^2
 =&\,\left|\int^{v^g}_{u^g}\left(\sigma'_\xi(s,y+h,\xi)-\sigma'_\xi(s,y,\xi)\right)d\xi\right|^2\\
 \lesssim&\,\left|u^g(s,y)-v^g(s,y)\right|^2\cdot \left|h\right|^2.
  \end{split}
  \end{equation}
 Thus,  \eqref{h1dayu} and  \eqref{h1xiaoyu} yield that
 \begin{equation}\label{v2label}
 \begin{split}
 V_3(t)
 \lesssim&\,\int_0^t\int_{\RR^2}\left|G(t-s,x-y)\right|^2|u^g(s,y)-v^g(s,y)|^2dydxds\\
 =&\,\int_0^t(t-s)\cdot S_1(s)ds.
 \end{split}
  \end{equation}
Collecting the inequalities in \eqref{4.84}, \eqref{v3label}, \eqref{v1label} and \eqref{v2label} together, we arrive at
 \begin{align}\label{s1label}
 S_1(t)\lesssim\, \int^t_0(t-s)^{2H}\cdot [S_1(s)+S_2(s)]ds.
 \end{align}
Applying the same procedure  as in the proof of \eqref{s1label}, we can show
 \begin{align}\label{s2label}
 S_2(t)\lesssim \,\int^t_0(t-s)^{4H-1}\cdot [S_1(s)+S_2(s)]ds.
 \end{align}
Therefore,  \eqref{s1label} and \eqref{s2label} imply that
\begin{align}\label{S1plusS2}
 S_1(t)+ S_2(t)\lesssim\, \int^t_0\left((t-s)^{2H}+(t-s)^{4H-1}\right)\cdot[S_1(s)+S_2(s)]ds.
 \end{align}
Invoking the fact that  $S_1(t)$ and $S_2(t)$ are uniformly bounded on $[0,T]$ by \eqref{S_1S_2_infty}, by the Gronwall inequality, we have
 $$S_1(t)+S_2(t)=0 \qquad  \text{ for   all } \quad t\in[0,T],$$
which together with the continuities of $u^g$ and $v^g$ gives that $u^g=v^g$.

 The proof is complete.
 \end{proof}

\section{Large deviation principle}
Firstly, we show that Condition \ref{cond1} (a) is  satisfied.  Recall that $\Gamma^{0}\left(\int_0^\cdot g(s)ds\right)=u^g$ for $g\in \mathbb{S}$, where $u^g$ is the solution of Eq.\,\eqref{eq skeleton}.
\begin{proposition}\label{thm continuity skeleton}
Assume that Hypothesis \ref{Hypoth.1} holds.
For every $N<+\infty$, let $\{g_n\}_{n\geq1}, g $ be in $S^N$ such that  $g_n\to g$ weakly as $n\to \infty$.
Then, as $n\rightarrow\infty$,
$$ u^{g_n}\longrightarrow  u^{g} \ \ \text{
in } \mathcal{C}([0, T]\times\RR),
$$
where $u^{g_n}$, $u^g$ are the solutions of Eq.\,\eqref{eq skeleton} associated with $g_n$ and $g$, respectively.
\end{proposition}
\begin{proof}
Recall Eq.\,\eqref{eq skeleton} with $g$ replaced by $g_n$,
 \begin{align}\label{ugn}
u^{g_n}(t,x)=I_0(t,x)+\int_0^t\langle G(t-s,x-\cdot)\sigma(s, \cdot,u^{g_n}(s, \cdot)), {g_n}(s,\cdot)\rangle_{\mathcal H}ds,\,\,t\geq 0,\,\,x\in\RR.
 \end{align}
 Since $\{u^{g_n}\}_{n\geq1}\subset S^N$, by Proposition \ref{thm solu skeleton}, we have that, for any $p> 2$,
 \begin{align}\label{ugn-z-normal}
 \sup_{n\geq1}\left\|u^{g_n}\right\|_{ Z^p(T)}<\infty.
\end{align}

Similarly to the H\"older continuity of $u^{\e}$ specified by \eqref{Hol1} and \eqref{Lemma-4.3-iii},
we can obtain the H\"older continuity of $u^{g_n}(t,x)$  (uniformly in $t\in[0,T]$ and $x\in\RR$).
Combining this with $u^{g_n}(0,0)=u_0(0)$, we know that $\{u^{g_n}\}_{n\geq1}$ is relatively compact on the space $({\mathcal C}([0,T]\times\RR), d_{\mathcal C})$ by the Arzel\`a-Ascoli theorem. Thus,
 there exists a subsequence of $\{u^{g_n}\}_{n\geq1}$ (still denoted by $\{u^{g_n}\}_{n\geq1}$) and $u\in\mathcal{C}([0,T]\times\RR)$ such that $u^{g_n}\rightarrow u$ as $n\rightarrow \infty$.
By \eqref{ugn-z-normal} and Lemma \ref{lemma3.3limits} (i), we have that,  for any $p>2$,
\begin{align}\label{and}
\|u\|_{ Z^p(T)}<\infty.
\end{align}

We now prove that $u=u^g$.
  Denote that
  \begin{align*}
  D_1(t)=\,\int_{\mathbb{R}}|u^{g_n}(t,x)-u^g(t,x)|^2dx
  \end{align*}
and
   \begin{align*}
  D_2(t)=\,\int_{\RR^2}\left|u^{g_n}(t,x+h)-u^g(t,x+h)-u^{g_n}(t,x)+u^g(t,x)\right|^2\cdot|h|^{2H-2}dhdx.
  \end{align*}
  According to \eqref{ugZlamdanormal}, we know that $\displaystyle\sup_{t\in[0,T]}D_1(t)<\infty$ and $\displaystyle\sup_{t\in[0,T]}D_2(t)<\infty$.
 Denote  $\triangle_n(t,x,y):=\sigma(t,x,u^{g_n}(t,y))-\sigma(t,x,u^g(t,y))$. Recall $\mathcal{D}_hf(t,x)$ defined by \eqref{TechDt}.
It follows from \eqref{eq skeleton} and \eqref{ugn} that
  \begin{equation}\label{D_1+D_2}
      \begin{split}
        &D_1(t)+D_2(t)\\
        \leq&\,2\int_{\RR}\left|\int^t_0\langle G(t-s,x-\cdot)\triangle_n(s, \cdot,\cdot), {g_n}(s,\cdot)\rangle_{\mathcal H}ds\right|^2dx\\
        &+\,2\int_{\RR}\left|\int^t_0\langle G(t-s,x-\cdot)\sigma(s,\cdot,u^g(s,\cdot)), {g_n}(s,\cdot)-g(s,\cdot)\rangle_{\mathcal H}ds\right|^2dx\\
        &+\,2\int_{\RR}\left|\int^t_0\langle \mathcal{D}_hG(t-s,x-\cdot)\triangle_n(s, \cdot,\cdot), {g_n}(s,\cdot)\rangle_{\mathcal H}ds\right|^2\cdot|h|^{2H-2}dhdx\\
        &+\,2\int_{\RR}\left|\int^t_0\langle\mathcal{ D}_hG(t-s,x-\cdot)\sigma(s,\cdot,u^g(s,\cdot)), {g_n}(s,\cdot)-g(s,\cdot)\rangle_{\mathcal H}ds\right|^2\cdot|h|^{2H-2}dhdx\\
        =:&\,2\left (E_1(t)+E_2(t)+E_3(t)+E_4(t)\right).
      \end{split}
  \end{equation}
  Using  the similar procedure as in the proof of \eqref{S1plusS2}, we can show that
   \begin{equation}\label{E_1+E_3}
      \begin{split}
        E_1(t)+E_3(t)
        \lesssim\,
   \int^t_0(t-s)^{4H-1}\cdot\left[D_1(s)+D_2(s)\right]ds.
      \end{split}
  \end{equation}
 On the other hand, denote that
 \begin{align*}
   F_n(t,x):=\left|\int^t_0\langle G(t-s,x-\cdot)\sigma(s,\cdot,u^g(s,\cdot)), {g_n}(s,\cdot)-g(s,\cdot)\rangle_{\mathcal H}ds\right|^2.
 \end{align*}
  By Lemma \ref{lemma3.3limits} (ii),
 we know that, for almost all $x\in \mathbb R$,
 \begin{align*}
    \int_{0}^{t}\left\| G(t-s,x-\cdot)\sigma(s,\cdot,u^{g}(s,\cdot))\right\|_{\HH}^2ds<\infty.
 \end{align*}
As $g_n$ converges  weakly   to $g$ when $n$ tends to $\infty$, we know that,  for  almost all $x\in \mathbb R$,
  \begin{align}\label{eq H G2}
      F_n(t,x) \rightarrow 0, \ \ \text{as }  n\rightarrow \infty.
 \end{align}
 Since $g_n, g\in S^N$,  by using the Cauchy-Schwarz inequality, \eqref{eq H product} and a change of variable, we have that, for any $p> 2$,
 \begin{align*}
       &\int_{\RR}  F_n(t,x) ^{\frac{p}{2}}dx\notag\\
      \lesssim&\,\int_{\RR}\lc\int_{0}^{t}\left\| G(t-s,x-\cdot)\sigma(s,\cdot,u^{g}(s,\cdot))\right\|_{\HH}^2ds\rc^{\frac p2}dx \notag\\
     \simeq&\,
       \int_{\RR}\bigg(\int_{0}^{t}\int_{\RR^2}  \Big|G(t-s,x-y-h)\sigma(s,y+h,u^{g}(s,y+h))\\
       &\qquad\qquad\quad-G(t-s,x-y)\sigma(s,y,u^{g}(s,y))\Big|^2 \cdot|h|^{2H-2}dhdyds\bigg)^{\frac p2}dx.\notag
     \end{align*}
  By the similar technique as that  in Step 2 in the proof of Lemma
   \ref{UniBExist}, we have
  \begin{align*}
       &\sup_{n\geq1}\int_{\RR}  F_n(t,x) ^{\frac{p}{2}}dx
      \lesssim \|u^g\|^p_{Z^p(T)}<\infty.
   \end{align*}
 It follows  that $\{ F_n(t,x)\}_{n\geq 1}$ is $L^1$-uniformly integrable in $(\mathbb R, dx)$.
By \eqref{eq H G2} and the uniform integrability convergence theorem,  we have
 \begin{align}\label{E2lim0}
   \lim_{n\rightarrow\infty} E_2(t)=0.
 \end{align}
Similarly, we can prove
  \begin{align}\label{E4lim0}
  \lim_{n\rightarrow\infty} E_4(t)=0.
 \end{align}
Putting \eqref{D_1+D_2}, \eqref{E_1+E_3}, \eqref{E2lim0} and \eqref{E4lim0} together, by the Gronwall lemma, we have $$\displaystyle\lim_{n\rightarrow\infty}\left[D_1(t)+D_2(t)\right]=0,  \, \text{for all }\, t\in[0,T].$$
 In particular, $u^{g_n}(t,\cdot)\rightarrow u^g(t,\cdot)$ as $n\rightarrow \infty$ in the space $ L^2(\RR)$ for all $t\in [0,T]$.
   Since
   $u^{g_n}$ also converges to $u$ as $n\rightarrow \infty$ in the space $\left(\mathcal{C}([0,T]\times\RR),d_{\mathcal{C}}\right)$, the uniqueness of the limit of $u^{g_n}$ implies that $u=u^g$.
 The proof is complete.
\end{proof}

We now verify the second part of Condition \ref{cond1}.
For any $\e\in(0,1)$, recall the  solution functional  $\Gamma^{\e}: C([0,T];\RR^{\infty}) \rightarrow \mathcal C([0,T]\times\RR) $ defined by
 \begin{align}\label{eq Gamma e}
 \Gamma^{\e}\left(W(\cdot)\right):=u^\e,
 \end{align}
  where $u^\e$ stands for the solution of Eq.\,\eqref{SWE}.

Let $\{g^{\e}\}_{\e\in(0,1)}\subset \mathcal U^N$ be a given family of stochastic processes.
By the Girsanov theorem, we know  that $\tilde{u}^{\e}:=\Gamma^{\e}\left(W({\cdot})+\frac{1}{\sqrt \e} \int_0^{\cdot}{g}^{\e}(s)ds \right)$ is the unique solution of
 \begin{equation}\label{eq SPDE Y-1}
\begin{split}
   \tilde{u}^{\e}(t,x)
  =&\,I_0(t,x)+\sqrt{\e}\int^t_0\int_{\RR}G(t-s,x-y)\sigma(s,y,\tilde{u}^{\e}(s,y)){W}(ds,dy) \\
  &+\int^t_0\left\langle G(t-s,x-\cdot)\sigma(s,\cdot,\tilde{u}^{\e}(s,\cdot)), g^{\e}(s,\cdot)\right\rangle_{\HH}ds.
\end{split}
\end{equation}
Moreover, the following stochastic  equation is held for $\bar{u}^{\e}:=\Gamma^0\left(\int_0^{\cdot} g^{\e}(s)ds\right)$, where $ \Gamma^0$ is defined by \eqref{eq Gamma0}:
\begin{align}\label{eq SPDE Z-1}
   \bar{u}^{\e}(t,x)
  =&\,I_0(t,x)+\int^t_0\left\langle G(t-s,x-\cdot)\sigma(s,\cdot,\bar{u}^{\e}(s,\cdot)), g^{\e}(s,\cdot)\right\rangle_{\HH}ds.
\end{align}
 \begin{proposition}
 \label{proposition-4-3}Assume that Hypothesis \ref{Hypoth.1} holds.  Then, for any family $\{g^{\e}\}_{\e\in(0,1)}\subset \mathcal U^N$ ($N<+\infty$) and for any $\delta>0$,
    $$\lim_{\e\rightarrow 0}\mathbb P\left(d_{\mathcal C}(\tilde{u}^{\e}, \bar{u}^{\e})>\delta\right)=0. $$
    \end{proposition}

Before proving Proposition \ref{proposition-4-3}, we give the following lemmas, which assert the finiteness of the $\|\cdot\|_{\mathcal{Z}^p(T)}$ norm and the H\"{o}lder continuity of $\tilde{u}^{\e}$ and $\bar{u}^{\e}$.
 \begin{lemma}\label{lemma-5-1}Assume that Hypothesis \ref{Hypoth.1} holds. Then, it holds that for any  $p>\frac{2}{4H-1}$,
   \begin{align} \label{4-95-bound}
   \sup_{\e\in(0,1)}\|\tilde{u}^{\e}\|_{\mathcal{Z}^p(T)}<\infty,\,\,\,\,\sup_{\e\in(0,1)}\|\bar{u}^{\e}\|_{\mathcal{Z}^p(T)}<\infty.
   \end{align}
   \end{lemma}
\begin{proof}
We provide the  detailed proof for the result concerning $\tilde{u}^{\e}$,
while the proof for that with respect to $\bar{u}^{\e}$ is similar but simpler and we omit it here.
Following the idea developed by the proof of \cite[Lemma 3.4]{LHW2022}, we approximate the noise $W$ by smoothing it with respect to the space variable. For each $\eta>0$ and $\varphi\in\HHH$, we define
\begin{align*}
W_\eta(\varphi)=&\int^t_0\int_{\RR}[\rho_\eta\ast \varphi](s,y)W(ds,dy)\\
=&\int^t_0\int_{\RR}\int_{\RR}\varphi(s,x)\rho_\eta(x-y)W(ds,dy)dx,
\end{align*}
 where $\rho_\eta(x)=\frac{1}{\sqrt{2\pi \eta}}\exp\left(-\frac{x^2}{2\eta}\right)$.
 Consider the equation
 \begin{equation}\label{tildeueetaDEF}
 \begin{split}
 \tilde{u}^{\e}_\eta(t,x)=&\,I_0(t,x)+\sqrt{\e}\int^t_0\int_{\RR}G(t-s,x-y)\sigma(s,y,\tilde{u}^{\e}_\eta(s,y)){W_\eta}(ds,dy)\\
 &+\,\int^t_0\left\langle G(t-s,x-\cdot)\sigma(s,\cdot,\tilde{u}^{\e}_\eta(s,\cdot)), g^{\e}(s,\cdot)\right\rangle_{\HH_\eta}ds\\
 =:&\,I_0(t,x)+\sqrt{\e}\Phi_{1,\eta}^{\e}(t,x)+\Phi_{2,\eta}^{\e}(t,x).
 \end{split}
 \end{equation}
 We introduce the following Picard iteration:
 \begin{equation}\label{tildeueetaDEFn}
 \begin{split}
 \tilde{u}^{\e,n+1}_\eta(t,x)=&\,I_0(t,x)+\sqrt{\e}\int^t_0\int_{\RR}G(t-s,x-y)\sigma(s,y,\tilde{u}^{\e,n}_\eta(s,y)){W_\eta}(ds,dy)\\
 &+\,\int^t_0\left\langle G(t-s,x-\cdot)\sigma(s,\cdot,\tilde{u}^{\e,n}_\eta(s,\cdot)), g^{\e}(s,\cdot)\right\rangle_{\HH_\eta}ds\\
 =:&\,I_0(t,x)+\sqrt{\e}\Phi_{1,\eta}^{\e,n}(t,x)+\Phi_{2,\eta}^{\e,n}(t,x),
 \end{split}
 \end{equation}
 with $\tilde{u}^{\e,0}_\eta(t,x)=I_0(t,x)$.
From the proof of  \cite[Lemma 3.5]{LHW2022} and by the similar argument to  that in Step 1 in the proof of Lemma \ref{UniBExist}, we know that,  for any fixed $t\in [0,T]$, $\eta>0$ and $\e\in(0,1)$, $\tilde{u}^{\e,n}_{\eta}(t,\cdot)$ converges to $\tilde{u}^{\e}_{\eta}(t,\cdot)$ in $L^p(\Omega\times\mathbb{R})$ when $n$ goes to infinity.

A similar computation as that in Steps 2 and 3 in the proof of \cite[Lemma 3.5]{LHW2022} gives that
\begin{equation}\label{eq u e521}
\begin{split}
&\|\Phi_{1,{\eta}}^{\e,n+1}(t,\cdot)\|^2_{L^p(\Omega\times\mathbb{R})}
+\left[\mathcal{N}^*_{\frac{1}{2}-H,\,p}\Phi_{1,{\eta}}^{\e,n+1}(s)\right]^2\\
\lesssim&\,\int^t_0\left((t-s)+(t-s)^{2H}+(t-s)^{4H-1}\right)\cdot \|\tilde{u}^{\e,n}_{\eta}(s,\cdot)\|^2_{L^p(\Omega\times\RR)}ds\\
&+\,\int^t_0\left((t-s)+(t-s)^{2H}\right)\cdot\left[\mathcal{N}^*_{\frac{1}{2}-H,\,p}\tilde{u}^{\e,n}_{\eta}(s)\right]^2ds.
\end{split}
\end{equation}
Furthermore, as in the proof of Lemma \ref{UniBExist}, we obtain that
\begin{equation}\label{eq u e52}
\begin{split}
&\|\Phi_{2,{\eta}}^{\e,n+1}(t,\cdot)\|^2_{L^p(\Omega\times\mathbb{R})}
+\left[\mathcal{N}^*_{\frac{1}{2}-H,\,p}\Phi^{\e,n+1}_{2,{\eta}}(t)\right]^2\\
\lesssim&\,\int^t_0\left((t-s)+(t-s)^{2H}+(t-s)^{4H-1}\right)\cdot \|\tilde{u}^{\e,n}_{\eta}(s,\cdot)\|^2_{L^p(\Omega\times\RR)}ds\\
&+\,\int^t_0\left((t-s)+(t-s)^{2H}\right)\cdot\left[\mathcal{N}^*_{\frac{1}{2}-H,\,p}\tilde{u}^{\e,n}_{\eta}(s)\right]^2ds.
\end{split}
\end{equation}
 For any $t\ge0$, let
 $$\tilde{\Psi}^{\e,n}_{\eta}(t):=\|\tilde{u}^{\e,n}_{\eta}(t,\cdot)\|^2_{L^p(\Omega\times\RR)}+\left[\mathcal{N}^*_{\frac{1}{2}-H,\,p}\tilde{u}^{\e,n}_{\eta}(t)\right]^2.$$
Putting \eqref{tildeueetaDEFn}, \eqref{eq u e521} and \eqref{eq u e52} together, there exists a constant
  $C_{T,p,H,N}>0$ such that
\begin{align*}
\tilde{\Psi}^{\e,n+1}_{\eta}(t)\leq \,C_{T,p,H,N}\left(\|I_0\|^2_{Z^p(T)}+\int^t_0(t-s)^{4H-1}\cdot \tilde{\Psi}^{\e,n}_{\eta}(s)ds\right).
\end{align*}
Hence, by the extension of Gronwall's lemma \cite[Lemma 15]{Dalang1999}, there exists   a constant $C$ independent of  $\eta$  and $\e$ such that
$$
\sup_{n\geq1}\sup_{t\in[0,T]}\tilde{\Psi}^{\e,n}_{\eta}(t)\leq\, C.
  $$
According to Fatou's lemma, we have that for any $p>\frac{2}{4H-1}$, there exists   a constant $C$  independent of $\e\in(0,1)$ such that
\begin{align}\label{eq Zp sec 5}
\displaystyle\sup_{{\eta}>0}\|\tilde{u}^{\e}_{\eta}\|_{\mathcal{Z}^p(T)}\leq C.
\end{align}

By the same methods as that in the proof of both  \cite[Lemma 4.2 (ii), (iii)]{LHW2022} and Lemma \ref{TimeSpaceRegBdd} (ii), (iii),
we can get the uniform H\"older continuity of $\Phi_{i,{\eta}}^{\e}$, $i=1,2$. This, together with the fact $\tilde{u}^{\e}_{\eta}(0,0)=u_0(0)$  implies  that the laws of the family $\{\tilde{u}^{\e}_{\eta}\}_{\eta>0}$ are tight  in the space $(\cC([0, T]\times\RR), d_{\cC})$.
Thus, $\tilde{u}^{\e}_{\eta}\rightarrow \tilde{u}^{\e}$ almost surely in the space $\left(\mathcal{C}([0,T]\times\RR),d_\mathcal{C}\right)$ as $\eta\rightarrow 0$. Furthermore, we can obtain that $\tilde{u}^{\e}$ is the solution of Eq.\,\eqref{eq SPDE Y-1} using the same method as that in the proofs of \cite[Theorem 1.5]{HW2019} and Proposition \ref{thm solu skeleton}. By \eqref{eq Zp sec 5} and \cite[Lemma 4.6]{HW2019},
we have that for any $p>\frac{2}{4H-1}$,
$\sup_{\e\in(0,1)}\|\tilde{u}^{\e}\|_{\mathcal{Z}^p(T)}<\infty$. The proof is complete.
\end{proof}

 For  any $u\in \mathcal{Z}^p(T)$ and $g\in \mathcal U^N$, let
\begin{align}\label{Phi2definition}
Y(t,x):=\,\int^t_0\left\langle G(t-s,x-\cdot)\sigma(s,\cdot,u(s,\cdot)), g(s,\cdot)\right\rangle_{\HH}ds.
\end{align}
By  Lemma \ref{TimeSpaceRegBdd} and Minkowski's
inequality, we have the following lemma.
\begin{lemma}\label{lemma-5.2}
 Assume that   Hypothesis \ref{Hypoth.1} holds.  Then we have the following results:
\begin{itemize}
\item[(i).]for any $p>\frac{2}{4H-1}$,  there exists a constant $ C_{T, p,H,N}>0$  such that
\begin{align} \label{4-96-bound}
\mathbb{E}\left[\left|\sup_{t\in[0,T],\,x\in\mathbb{R}}
\mathcal{N}_{\frac{1}{2}-H}Y(t,x)\right|^p\right]\leq \, C_{T,p,H,N} \left\|u\right\|^p_{\mathcal{Z}^p(T)};
\end{align}
    \item[(ii).] if $p>\frac{1}{H}$ and
     $0<\gamma< H-\frac{1}{p}$,, then there exists a positive constant {\bf$C_{T,p,H,N,\gamma}$} such that
  \begin{equation}\label{eq lem51}
  \begin{split}
   \mathbb{E}\left[\left|\sup\limits_{\substack{t,\,t+h\in[0,T], \\ x\in \RR}}\big[Y(t+h,x)  -Y(t,x) \big]\right|^p\right]
   \leq  \,C_{T,p,H,N,\gamma}|h|^{p\gamma} \cdot \left\|u\right\|^p_{\mathcal{Z}^p(T)};
   \end{split}
   \end{equation}
   \item[(iii).] if {\bf $p>\frac{1}{H}$ }and
   {\bf$0<\gamma<H-\frac{1}{p}$}, then  there exists a positive constant {\bf$C_{T,p,H,N,\gamma}$} such that
   \begin{equation}\label{eq lem52}
   \begin{split}
  \mathbb{E}\left[\left|\sup\limits_{\substack{t\in[0,T], \\ x,y\in \RR}}[Y(t,x)- Y(t,y)] \right|^p\right]
   \leq   \,C_{ T,p,H,N,\gamma}|x-y|^{p\gamma} \cdot  \left\|u\right\|^p_{\mathcal{Z}^p(T)}.
   \end{split}
   \end{equation}
\end{itemize}
\end{lemma}

 Recall  $\tilde u^{\e}$ and $\bar u^{\e}$ defined by \eqref{eq SPDE Y-1} and \eqref{eq SPDE Z-1}, respectively. For any $k\geq 1$ and  $p>\frac{2}{4H-1}$,  define the  stopping time
\begin{equation}\label{Stoptime1}
\begin{split}
\tau_k:=\inf\bigg\{&r\ge 0:\,\,\sup_{0\leq s\leq r,\,x\in\mathbb{R}}
\mathcal{N}_{\frac{1}{2}-H}\tilde{u}^{\e}(s,x)\geq k,\,\,  \text{or }\,\,\sup_{0\leq s\leq r,\,x\in\mathbb{R}}
\mathcal{N}_{\frac{1}{2}-H}\bar{u}^{\e}(s,x)\geq k\bigg\}.
\end{split}
\end{equation}
By \cite[Lemma 4.2]{LHW2022},  Lemmas  \ref{lemma-5-1} and \ref{lemma-5.2}, we know that
$$
\sup_{\e\in (0,1)}\left\|\sup_{t\in[0,T],\,x\in\mathbb{R}}
\mathcal{N}_{\frac{1}{2}-H}\tilde u^{\e}(t,x)\right\|_{L^p(\Omega)}<\infty, \  \ \
\sup_{\e\in (0,1)}\left\|\sup_{t\in[0,T],\,x\in\mathbb{R}}
\mathcal{N}_{\frac{1}{2}-H}\bar u^{\e}(t,x)\right\|_{L^p(\Omega)}<\infty,
$$
which, together with Markov's inequality, implies that
\begin{equation}\label{stop-infinit}
\tau_k\uparrow \infty,\text{ a.s., as } k\rightarrow \infty.
\end{equation}
For any $t\in [0,T]$, let
\begin{align}
\tilde{u}^{\e}_k(t,\cdot):=\tilde{u}^{\e}(t\wedge \tau_k,\cdot),\ \ \ \ \ \bar{u}^{\e}_k(t,\cdot):=\bar{u}^{\e}(t\wedge \tau_k,\cdot).
\end{align}
 It is easy to see that, for any $t\in [0,T]$ ,
$\tilde{u}^{\e}_k(t,\cdot)=\tilde{u}^{\e}(t,\cdot)$ and  $\bar{u}^{\e}_k(t,\cdot)=\bar{u}^{\e}(t,\cdot)$ when $\tau_k> T$.
\begin{lemma}\label{utildebar}
For any $\e>0$,
\begin{align}\label{u-d-c-E}
\lim_{\e\rightarrow0}\big\|\tilde{u}_k^{\e}-\bar{u}_k^{\e}\big\|_{\mathcal{Z}^p(T)}=0.
\end{align}
\end{lemma}
\begin{proof}
Let
\begin{align}\label{Phi1definition}
\Phi_{1}^{\e}(t,x):=\,\int^{t }_0\int_{\RR}G(t-s,x-y)\sigma(s,y,\tilde{u}_k^{\e}(s,y)){W}(ds,dy)
\end{align}
 and
\begin{equation}\label{Phi3definition}
\begin{split}
 \Phi_2^{\e}(t,x):=\int^{t}_0\left\langle G(t-s,x-\cdot)\triangle(s,\cdot,\cdot). g^{\e}(s,\cdot)\right\rangle_{\HH}ds.
\end{split}
\end{equation}
where  $\triangle(t,x,y):=\sigma(t,x,\tilde{u}^{\e}_k(t,y))-\sigma(t,x,\bar{u}^{\e}_k(t,y))$.
Then, by \eqref{eq SPDE Y-1} and \eqref{eq SPDE Z-1}, we have
\begin{align}\label{5128-inequality}
\big\|\tilde{u}_k^{\e}-\bar{u}_k^{\e}\big\|_{\mathcal{Z}^p(T)}\leq\,\sqrt{\e}
\big\|\Phi_1^{\e}\big\|_{\mathcal{Z}^p(T)}+\big\|\Phi_2^{\e}\big\|_{\mathcal{Z}^p(T)}.
\end{align}
By using the same technique as in the proof of Lemmas 4.5 and  4.6   in \cite{HW2019}, we have $\big\|\Phi_1^{\e}\big\|_{\mathcal{Z}^p(T)}<\infty$ for any $p>\frac{2}{4H-1}$.
By   \eqref{Lipschitzianuniformlycondu} and a similar computation  as  that   in Step 2 in the proof of Lemma \ref{UniBExist}, we have
\begin{equation}\label{5139}
\begin{split}
\left\| \Phi_2^{\e}(t,\cdot)\right\|^2_{L^p(\Omega\times\RR)}
\lesssim&\,
\int^t_0(t-s)^{2H}\cdot
\big\|  \tilde{u}_k^{\e}(s,\cdot)-\bar{u}_k^{\e}(s,\cdot) \big\|_{L^p(\Omega\times\RR)}^2ds\\
&+\int^t_0(t-s)\cdot \left[\mathcal{N}^*_{\frac{1}{2}-H,\,p}
\left( \tilde{u}_k^{\e}(s)-\bar{u}_k^{\e}(s)\right)\right]^2ds.
\end{split}
\end{equation}
For the term $\mathcal{N}^*_{\frac{1}{2}-H,p}\Phi_2^{\e}(t)$, we assert that
\begin{equation}\label{5147}
\begin{split}
&\left[\mathcal{N}^*_{\frac{1}{2}-H,p}
\Phi_2^{\e}(t)\right]^2\\
\lesssim&\,k^2 \int^t_0\left((t-s)^{2H}+(t-s)^{4H-1}\right)\cdot \left\|\tilde{u}_k^{\e}(s,\cdot)-\bar{u}_k^{\e}(s,\cdot)\right\|^2_{L^p(\Omega\times\RR)}ds\\
&+\int^t_0(t-s)^{2H}\cdot \left[\mathcal{N}^*_{\frac{1}{2}-H,\,p}
\left(\tilde{u}_k^{\e}(s)-\bar{u}_k^{\e}(s)\right)\right]^2ds .
\end{split}
\end{equation}
Putting  \eqref{5128-inequality}, \eqref{5139} and \eqref{5147} together,  and by the  Gronwall inequality, we have that, for any fixed $k\geq 1$ and $p>\frac{2}{4H-1}$,
 \begin{align*}
  \lim_{\e\rightarrow0}\left\|\tilde{u}_k^{\e}-\bar{u}_k^{\e}\right\|_{\mathcal{Z}^p(T)}=0.
\end{align*}

It remains to prove \eqref{5147}.
Since $g^{\e}\in \mathcal{U}^N$, by the Cauchy-Schwarz inequality and \eqref{eq H product}, we have
\begin{align}\label{Phi_3ePhi_3etxh}
&\mathbb{E}\left[\big|\Phi_2^{\e}(t,x+h)-\Phi_2^{\e}(t,x)\big|^p\right]\notag \\
=&\,\mathbb{E}\left[\Big|\int^t_0\big\langle \mathcal{D}_{h}G(t-s,x-\cdot)\triangle(s,\cdot,\cdot), g^{\e}(s,\cdot)\big\rangle_{\HH}ds\Big|^p\right] \notag\\
\lesssim&\,\mathbb{E}\Bigg[\int^t_0\int_{\RR^2}\big| \mathcal{D}_{h}G(t-s,x-y-z)\triangle(s,y+z,y+z)\\
&\qquad\qquad\quad-\mathcal{D}_{h}G(t-s,x-y)\triangle(s,y,y)\big|^2\cdot |z|^{2H-2}dzdyds\Bigg]^{\frac{p}{2}}\notag \\
\lesssim&\,\mathcal{R}_1^{\e}(t,x,h)+\mathcal{R}_2^{\e}(t,x,h)+\mathcal{R}_3^{\e}(t,x,h),\notag
\end{align}
where
\begin{align*}
\mathcal{R}_1^{\e}(t,x,h):=&\,\mathbb{E}\Bigg[\int^t_0\int_{\RR^2}\big| \mathcal{D}_{h}G(t-s,x-y-z)\big|^2\\
&\qquad\qquad \cdot\big|\triangle(s,y+z,y+z)
 -\triangle(s,y+z,y)\big|^2\cdot |z|^{2H-2}dzdyds\Bigg]^{\frac{p}{2}};\\
\mathcal{R}_2^{\e}(t,x,h):=&\,\mathbb{E}\Bigg[\int^t_0\int_{\RR^2}\big| \Box_{h,-z}G(t-s,x-y)\big|^2\cdot \big|\triangle(s,y+z,y)\big|^2\cdot|z|^{2H-2}dzdyds\Bigg]^{\frac{p}{2}};\\
\mathcal{R}_3^{\e}(t,x,h):=&\,\mathbb{E}\Bigg[\int^t_0\int_{\RR^2}\big|  \mathcal{D}_{h}G(t-s,x-y)\big|^2\cdot \big|\triangle(s,y+z,y)-\triangle(s,y,y)\big|^2 \cdot|z|^{2H-2}dzdyds\Bigg]^{\frac{p}{2}}.
\end{align*}
 By a change of variable and \eqref{490-490}, we have
 \begin{align*}
&\mathcal{R}_1^{\e}(t,x,h)\\
 \lesssim&\,\mathbb{E}\Bigg[\int^t_0\int_{\RR^2}\big| \mathcal{D}_{h}G(t-s,x-y)\big|^2\cdot |z|^{2H-2}dz\\
&\qquad \cdot\bigg[\big|\tilde{u}_k^{\e}(s,y+z)-\bar{u}_k^{\e}(s,y+z)
 -\tilde{u}_k^{\e}(s,y)+\bar{u}_k^{\e}(s,y)\big|^2\\
 &\qquad\quad+\,\big|\tilde{u}_k^{\e}(s,y)-\bar{u}_k^{\e}(s,y)\big|^2
\cdot \left(\big|\tilde{u}_k^{\e}(s,y+z)-\tilde{u}_k^{\e}(s,y)\big|^2+\big|\bar{u}_k^{\e}(s,y+z)-\bar{u}_k^{\e}(s,y)\big|^2\right)\bigg]
dyds\Bigg]^{\frac{p}{2}}.
 \end{align*}
Consequently,  we have
\begin{equation}\label{J21-bound}
\begin{split}
&\int_{\RR}\left(\int_{\RR}\mathcal{R}_{1}^{\e}(t,x,h)dx\right)^{\frac{2}{p}}\cdot |h|^{2H-2}dh \\
\lesssim&\,\int^t_0\int_{\RR^2}\big| \mathcal{D}_{h}G(t-s,y)\big|^2\cdot |h|^{2H-2}dhdy
\cdot
\Bigg[k^2\left(\int_{\RR}\mathbb{E}\left[\left|\tilde{u}_k^{\e}(s,x)-\bar{u}_k^{\e}(s,x)\right|^p\right]dx\right)^{\frac{2}{p}}\\
&\qquad+\int_{\RR}\left(\int_{\RR}\mathbb{E}\big[\big|\tilde{u}_k^{\e}(s,x+z)-\bar{u}_k^{\e}(s,x+z)
 -\tilde{u}_k^{\e}(s,x)+\bar{u}_k^{\e}(s,x)\big|^p\big]dx\right)^{\frac{2}{p}}\cdot |z|^{2H-2}dz \Bigg]ds\\
\lesssim&\,\int^t_0(t-s)^{2H}\cdot \left[k^2 \cdot \left\|\tilde{u}_k^{\e}(s,\cdot)-\bar{u}_k^{\e}(s,\cdot)\right\|^2_{L^p(\Omega\times\RR)}+\left[\mathcal{N}^*_{\frac{1}{2}-H,\,p}
\left(\tilde{u}_k^{\e}(s)-\bar{u}_k^{\e}(s)\right)\right]^2\right]ds,
\end{split}
\end{equation}
where  a change of variable, the Minkowski inequality and Lemma \ref{TechLemma3} have been used.
Invoking the Lipschitz continuity of $\sigma$, a change of variable, Minkowski's inequality and Lemma \ref{TechLemma3}, we have
\begin{equation}\label{J222-bound}
\begin{split}
&\int_{\RR}\left(\int_{\RR}\mathcal{R}_{2}^{\e}(t,x,h)dx\right)^{\frac{2}{p}}\cdot |h|^{2H-2}dh \\
\lesssim&\,\int^t_0\int_{\RR^3}\big| \Box_{h,-z}G(t-s,y)\big|^2\cdot|z|^{2H-2}\cdot|h|^{2H-2}dzdhdy
\cdot\left(\int_{\RR}\mathbb{E}\left[\left|\tilde{u}_k^{\e}(s,x)-\bar{u}_k^{\e}(s,x)\right|^p\right]dx\right)^{\frac{2}{p}}ds\\
\lesssim&\,\int^t_0(t-s)^{4H-1}\cdot \left\|\tilde{u}_k^{\e}(s,\cdot)-\bar{u}_k^{\e}(s,\cdot)\right\|^2_{L^p(\Omega\times\RR)}ds.
\end{split}
\end{equation}
It follows from the same analysis to \eqref{h1dayu} and  \eqref{h1xiaoyu} and Lemma \ref{TechLemma3} that
\begin{align}\label{5-155R1}
&\int_{\RR}\left(\int_{\RR}\mathcal{R}_{3}^{\e}(t,x,h)\right)^{\frac{2}{p}}\cdot |h|^{2H-2}dh\notag \\
\lesssim&\,\int^t_0\int_{\RR^2}\big|  \mathcal{D}_{h}G(t-s,y)\big|^2\cdot|h|^{2H-2}dhdy
\cdot\left(\int_{\RR}\mathbb{E}\left[\left|\tilde{u}_k^{\e}(s,x)-\bar{u}_k^{\e}(s,x)\right|^p\right]dx\right)^{\frac{2}{p}}ds \\
\lesssim&\,\int^t_0(t-s)^{2H}\cdot \left\|\tilde{u}_k^{\e}(s,\cdot)-\bar{u}_k^{\e}(s,\cdot)\right\|^2_{L^p(\Omega\times\RR)}ds.\notag
\end{align}
Putting \eqref{Phi_3ePhi_3etxh}-\eqref{5-155R1} together, we have \eqref{5147}. The proof is complete.
\end{proof}

We now prove Proposition \ref{proposition-4-3}.
\begin{proof}[Proof of Proposition \ref{proposition-4-3}]
 By \cite[Lemma 4.2]{LHW2022}, Lemmas  \ref{lemma-5-1} and  \ref{lemma-5.2}, we know that  the probability measures on the space $(\mathcal{C}([0,T]\times\RR),\mathcal{B}(\mathcal{C}([0,T]\times\RR)), d_{\mathcal{C}})$ corresponding to the processes    $\{\tilde{u}^{\e}-\bar{u}^{\e}\}_{\e>0}$   are tight. Thus, there is a subsequence $\e_n\downarrow 0$ such that $\tilde{u}^{\e_n}-\bar{u}^{\e_n}$ converges weakly
  to some stochastic process $Z=\{Z(t,x),  t\in[0,T], x\in \mathbb R\}$ in  $(\mathcal{C}([0,T]\times\RR),\mathcal{B}(\mathcal{C}([0,T]\times\RR)), d_{\mathcal{C}})$.

 On the other hand, for any $k\ge1, p>\frac{2}{4H-1}, \gamma>0$,
\begin{equation}\label{eq P}
\begin{split}
& \mathbb P\left(\sup_{0\le t\le T}\int_{\mathbb  R}\left|\tilde{u}^{\e}(t,x)-\bar{u}^{\e}(t,x)\right|^p dx >\gamma\right)\\
\le\, & \mathbb P\left(\sup_{0\le t\le T}\int_{\mathbb  R}\left|\tilde{u}^{\e}(t,x)-\bar{u}^{\e}(t,x)\right|^pdx >\gamma, \tau_k> T\right)
 +\mathbb P\left(\tau_k\leq T\right)\\
 \le \, &\mathbb P\left(\sup_{0\le t\le T}\int_{\mathbb  R}\left|\tilde{u}_k^{\e}(t,x)-\bar{u}_k^{\e}(t,x)\right|^p dx >\gamma\right)
 +\mathbb P\left(\tau_k\leq T\right).
\end{split}
\end{equation}
First   letting $\e\rightarrow0$ and then letting  $k\rightarrow\infty$,    by  Lemma \ref{utildebar} and \eqref{stop-infinit}, we have
 $$
\sup_{0\le t\le T}\int_{\mathbb  R}\left|\tilde{u}^{\e}(t,x)-\bar{u}^{\e}(t,x)\right|^p dx   \rightarrow 0 \text{ in probability,\,\,\,\,  as }\,  \e\rightarrow0.
 $$
 Thus, for any fixed $t\in [0,T]$, the processes  $\{\tilde{u}^{\e}(t, x)(\omega)-\bar{u}^{\e}(t,x)(\omega); (\omega,x)\in \Omega\times\mathbb R\}$ converges to 0 in probability in the product probability space $( \Omega\otimes\mathbb R, \mathbb P\otimes dx )$. This, together with the weak convergence  of $\{\tilde{u}^{\e_n}-\bar{u}^{\e_n}\}_{n\ge1}$ and the uniqueness of the limit distribution  of  $\tilde{u}^{\e_n}-\bar{u}^{\e_n}$, implies that   $Z(t,x)\equiv0, t\in [0,T], x\in \mathbb R$, almost surely.
Thus,  as $\e\rightarrow0$,  $\tilde{u}^{\e}-\bar{u}^{\e}$ converges weakly to  $0$ in $(\mathcal{C}([0,T]\times\RR),\mathcal{B}(\mathcal{C}([0,T]\times\RR)), d_{\mathcal{C}})$. Equivalently,
  the sequence of real-valued random variables
$   d_\mathcal{C}(\tilde{u}^{\e},\bar{u}^{\e})$ converges to $0$  in distribution as $\e$ goes to $0$.
It follows from \cite[p. 98, Exercise 4]{Chung} that
\begin{align*}  d_\mathcal{C}(\tilde{u}^{\e},\bar{u}^{\e})\longrightarrow 0  \text{ in probability},\,\,\, \text{as}\,\,\e\rightarrow0. \end{align*}
The proof is complete.
\end{proof}

\section{Appendix}

 In this section, we first give some lemmas  related to the wave Green's function $G(t,x)$ from \cite{LHW2022}, which play a key role for the proofs in this paper. Then we present a lemma providing  the H\"older continuity of $I_0(t,x)$.  Lastly, we give an  auxiliary  lemma used in the  proof of Proposition \ref{thm solu skeleton}.

\begin{lemma}\label{3.1wavekernelexp}(\cite[Lemma 3.1]{LHW2022}, \cite[Remark 3.3]{LHW2022})
The wave kernel $G(t,x)=\frac{1}{2}\textbf{1}_{\{|x|<t\}}$ can be expressed as
\begin{equation}\label{eq.SumKer}
  \begin{split}
  G(t-s,x-y)=&  \int_{\RR} \cC_{\beta}(t-r,x-z)\cS_{1-\beta}(r-s,z-y)dz \\
 &+\int_{\RR} \cS_{\alpha}(t-r,x-z)\cC_{1-\alpha}(r-s,z-y)dz \\
&+\int_{\RR} \cS(t-r,x-z)\mathcal{E}(r-s,z-y)dz \\
&+\int_{\RR} \mathcal{E}(t-r,x-z)\cS(r-s,z-y)dz,
  \end{split}
\end{equation}
where $\alpha,\beta\in(0,1)$, $\cS(t,x)=\cS_{1}(t,x)=G(t,x)=\frac 12 \1_{\{|x|<t\}}$ and
\begin{equation}\label{eq.cC_alpha}
	\begin{cases}
		\mathcal{E}(t,x) :=\frac{1}{\pi} \frac{t}{t^2+x^2}\,, \\
  \cS_{\alpha}(t,x) :=\frac{\Gamma(1-\alpha)}{2\pi}\cos\lc\frac{\alpha\pi}{2}\rc \lk(t+|x|)^{\alpha-1}+\hbox{\rm sgn}(t-|x|)\big|t-|x|\big|^{\alpha-1}\rk\,, \\
  \cC_{1-\alpha}(t,x) :=\frac{\Gamma(\alpha)}{2\pi}\bigg[\cos\lc\frac{\alpha\pi}{2}\rc\blk \big|t+|x|\big|^{-\alpha}+\big|t-|x|\big|^{-\alpha}\brk\\
 \qquad\qquad\qquad\qquad -2 \cos\lc \alpha\tan^{-1}\lc\frac{|x|}{t}\rc\rc [t^2+x^2]^{-\frac{\alpha}{2}}\bigg]\,.
	\end{cases}
\end{equation}
\end{lemma}

Recall $\mathcal{D}_{h}f(t,x)$, $\Box_{y,h}f(t,x)$ defined by \eqref{TechDt} and \eqref{TechBoxt}, respectively.
\begin{lemma}(\cite[(3.33), (3.34)]{LHW2022})\label{TechLemma3}
For any $H\in (\frac{1}{4},\frac{1}{2})$ and any $t>0$, there exists some constant $C_H$ such that
\begin{align}\label{Hu-3-32}
\int_{\RR^2}|\mathcal{D}_{h}G(t,x)|^2\cdot |h|^{2H-2}dhdx\leq C_H t^{2H}
\end{align}
and
\begin{align}\label{Hu-3-33}
\int_{\RR^3}|\Box_{y,h}G(t,x)|^2 |y|^{2H-2}dy|h|^{2H-2}dhdx\leq C_H t^{4H-1}.
\end{align}
\end{lemma}
For convenience,
 we denote that \begin{align}\label{mathcalK_def}
 \mathcal{K}_1=\mathcal{C}_\alpha,\quad \mathcal{K}_2=\mathcal{S}_\alpha,\quad\mathcal{K}_3=\mathcal{S},\quad
 \mathcal{K}_4=\mathcal{E}.
 \end{align}
 According to \eqref{eq.SumKer},  we define $\mathcal{\bar{K}}_i$ to be the complements of $\mathcal{K}_i$, $ i=1,2,3, 4$,  i.e.,
\begin{align}\label{mathcalbarK_def}
 \mathcal{\bar{K}}_1=\mathcal{S}_{1-\alpha},\quad \mathcal{\bar{K}}_2=\mathcal{C}_{1-\alpha},\quad\mathcal{\bar{K}}_3=\mathcal{E},\quad
 \mathcal{\bar{K}}_4=\mathcal{S}.
 \end{align}
Let
  \begin{align}\label{461-transform 0}
     \Phi_\e^g(t,x)
     := \,\int_{0}^{t}\langle G(t-s,x-\cdot)\sigma(s,\cdot,u_\e^g(s,\cdot)), g(s,\cdot) \rangle_{\HH_\e}ds.
\end{align}
where $\langle\cdot,\cdot\rangle_{\HH_\e}$ is defined by \eqref{eq product H e}.
Then,  by a stochastic version of Fubini's theorem and Lemma \ref{3.1wavekernelexp}, we have
that, for any $\theta\in(0,1)$,
    \begin{align}\label{461-transform}
     \Phi_\e^g(t,x)
     \simeq&\,\sum_{i=1}^4\frac{\sin(\theta\pi)}{\pi}\int_{0}^{t}\int_{\RR}(t-r)^{\theta-1}\mathcal{\bar{K}}_i(t-r,x-z)J_\theta^{\mathcal{K}_i}(r,z)dzdr,
 \end{align}
where
   \begin{align}\label{Def_JthetaK_i}
 J_\theta^{\mathcal{K}_i}(r,z)
 :=\int_0^r(r-s)^{-\theta}\langle \mathcal{K}_i(r-s,z-\cdot)\sigma(s,\cdot,u_\e^g(s,\cdot)), g(s,\cdot) \rangle_{\HH_\e}ds,
  \ \ \ \ \  i=1,2,3, 4.
 \end{align}
We give the following two lemmas related to $J_\theta^{\mathcal{K}_i}(r,z)$, $i=1, 2, 3, 4$, for proving Lemma \ref{TimeSpaceRegBdd}.
\begin{lemma}\label{LEMMA3.4}
 Assume that $\sigma(t,x,y)$ satisfies Hypothesis \ref{Hypoth.1}.  If $p>\frac{1}{H}$, $1-H<\alpha<1-\frac{1}{p}$ and $1-\frac{2}{q}+\alpha<\theta<H+\alpha-\frac{1}{2}$, then there exists some constant $C$, independent of $r\in[0,T]$, such that
 \begin{align}\label{lemma3.4est}
 \int_\RR\left|J_\theta^{\mathcal{K}_i}(r,z)\right|^pdz \leq C\left\|u_\e^g\right\|^p_{Z^p(T)}, \ \ \ \ \  i=1, 2, 3, 4.
 \end{align}
 \end{lemma}
 \begin{proof}
 For notational simplicity, we assume $\sigma(t,x,u)=\sigma(u)$ without loss of generality because of Hypothesis \ref{Hypoth.1}. The proof is similar as that of  \cite[Lemma B.1]{LHW2022} by replacing the stochastic integral by the deterministic integral.  More precisely, before we use the argument in the proof of  \cite[Lemma B.1]{LHW2022}, we need to use  \eqref{Def_JthetaK_i} and the similar technique as that in \eqref{uepLp} to have
  \begin{align*}
 \int_\RR\left|J_\theta^{\mathcal{K}_i}(r,z)\right|^pdz
 \lesssim\int_\RR\left[I_1(r,z,h)+I_2(r,z,h)\right]^{\frac{p}{2}}dz,  \ \ \ \ \  i=1,2,3, 4,
 \end{align*}
 where
 \begin{align*}
  I_1(r,z,h)=&\,\int_0^r\int_{\RR^2}(r-s)^{-2\theta}\big|\mathcal{K}_i(r-s,z-y-h)\big|^2\cdot\big|\mathcal{D}_hu_\e^g(s,y)\big|^2\cdot|h|^{2H-2}dydhds,\\
 I_2(r,z,h)=&\,\int_0^r\int_{\RR^2}(r-s)^{-2\theta}\big|\mathcal{D}_h\mathcal{K}_i(r-s,z-y)\big|^2\cdot\big|u_\e^g(s,y)\big|^2\cdot|h|^{2H-2}dydhds.
 \end{align*}
   The details are omitted here. The proof is complete.
 \end{proof}
The following lemma can be proved by the similar technique as that in the proofs of Lemma \ref{LEMMA3.4} and \cite[Lemma B.2]{LHW2022}. Here the proof is omitted.
 \begin{lemma}\label{D_hJtheKestimate}
Assume that $\sigma(t,x,y)$ satisfies the  hypothesis $(\mathbf{H})$. If $p>\frac{1}{H}$, $\frac{3}{2}-2H<\alpha<1-\frac{1}{p}$ and $1-\frac{2}{q}+\alpha<\theta<2H+\alpha-1$, then there exists a positive constant $C$, independent of $r\in[0,T]$, such that
\begin{align}\label{4-21}
\int_\RR\left[\int_\RR|J_\theta^{\mathcal{K}_i}(r,z+h)-J_\theta^{\mathcal{K}_i}(r,z)|^2\cdot|h|^{2H-2}dh\right]^{\frac{p}{2}}dz
\leq C\left\|u_\e^g\right\|^p_{Z^p(T)}, \ \ \ \ \ i=1,2,3, 4.
 \end{align}
 \end{lemma}
 For the H\"older continuity of $I_0(t,x)$, we have the following lemma, which follows from the proof of Theorem 1.1 (1) in  \cite{HLS2022}.
 \begin{lemma}\label{Holder_continu_u0v0}
 Assume that $u_0$ and $v_0$ are $\alpha$-H\"older continuous with $\alpha\in(0,1]$. Then we have
  \begin{align*}
 \left|I_0(t,x)-I_0(s,y)\right|\lesssim |t-s|^{\alpha}+|x-y|^{\alpha},\,t,s\in[0,T],x,y\in\RR.
   \end{align*}
 \end{lemma}
We also need the following lemma to prove the existence of the solution to the skeleton equation \eqref{eq skeleton}, whose proof is similar as that of \cite[Lemma 3.3]{LIWangZhang2022} and is omitted here.
\begin{lemma}\label{lemma3.3limits} Assume that     $u_{n}\in Z^p(T)$, $n\geq 1$,  for some $p\ge2$. If  $u_n\rightarrow u$ in $(\cC([0, T]\times\RR),d_{\cC})$ as $n\rightarrow \infty$, then we have the following results:
\begin{itemize}
     \item[(i).]     $u$ is also in $Z^p(T)$;
  \item[(ii).] for any fixed $t\in [0,T]$,
  \begin{align}\label{eq H1 finite}
  \int^t_0\int_{\mathbb R}\|G(t-s,x-\cdot)\sigma(s,\cdot,u(s,\cdot))\|^2_{\cH} dsdx<\infty.
  \end{align}
  Hence, for almost all  $x\in \mathbb R$,
    \begin{align}\label{eq H2 finite}
     \int^t_0\|G(t-s,x-\cdot)\sigma(s,\cdot,u(s,\cdot))\|^2_{\cH}  ds<\infty.
      \end{align}
  \end{itemize}
\end{lemma}

\vskip0.5cm
\noindent{\bf Acknowledgments } The research of R. Li is partially supported by Shanghai Sailing Program grant 21YF1415300 and NNSFC grant 12101392.
  The research of B. Zhang is
partially supported by NNSFC grants 11971361 and 11731012.

\end{document}